\documentclass{amsproc}

\usepackage{verbatim} 
\usepackage{amsmath}
\usepackage{amsfonts}
\usepackage{amssymb}
\usepackage{mathrsfs}
\usepackage{amsthm}
\usepackage{newlfont}

\usepackage{fullpage,amsmath, amssymb,amscd}
\usepackage{latexsym, epsfig, color}
\usepackage[mathscr]{eucal}
\usepackage{xypic, graphicx}
\usepackage{amscd}
\usepackage[all, cmtip]{xy}

\usepackage{datetime}

\usepackage{float, hyperref, nccmath, multirow, caption, subcaption}

\newcommand\cyr{%
 \renewcommand\rmdefault{wncyr}%
 \renewcommand\sfdefault{wncyss}%
 \renewcommand\encodingdefault{OT2}%
\normalfont\selectfont} \DeclareTextFontCommand{\textcyr}{\cyr}

\newtheorem{theorem}{Theorem}
\newtheorem{lemma}[theorem]{Lemma}
\newtheorem{corollary}[theorem]{Corollary}
\newtheorem{proposition}[theorem]{Proposition}
\newtheorem{remark}[theorem]{Remark}

\def\Z{\mathbb Z}
\def\N{\mathbb N}
\def\Q{\mathbb Q}
\def\R{\mathbb R}
\def\C{\mathbb C}
\def\F{\mathbb F}

\def\cA{\mathcal A}
\def\cP{\mathcal P}

\def\Ker{\operatorname{Ker}}

\def\Re{\operatorname{Re}}

\def\det{\operatorname{det}}
\def\tr{\operatorname{tr}}

\def\Gal{\operatorname{Gal}}

\def\Frob{\operatorname{Frob}}

\def\mod{\operatorname{mod}}

\def\disc{\operatorname{disc}}

\def\exp{\operatorname{exp}}

\def\gcd{\operatorname{gcd}}
\def\e{\operatorname{e}}

\def\GL{\operatorname{GL}}
\def\GSp{\operatorname{GSp}}

\def\PGL{\operatorname{PGL}}

\def\O{\operatorname{O}}
\def\o{\operatorname{o}}

\def\log{\operatorname{log}}

\def\e{\varepsilon}

\def\ds{\displaystyle}

\def\Scal{\operatorname{Scal}}

\begin{document}

\title{
Non-CM elliptic curves with infinitely many almost prime Frobenius traces
}

\date{July 11, 2020}



\author{
Alina Carmen Cojocaru
}
\address[Alina Carmen  Cojocaru]{
\begin{itemize}
\item[-]
Department of Mathematics, Statistics and Computer Science, University of Illinois at Chicago, 851 S Morgan St, 322
SEO, Chicago, 60607, IL, USA;
\item[-]
Institute of Mathematics  ``Simion Stoilow'' of the Romanian Academy, 21 Calea Grivitei St, Bucharest, 010702,
Sector 1, Romania
\end{itemize}
} \email[Alina Carmen  Cojocaru]{cojocaru@uic.edu}

\author{
McKinley Meyer}
\address[McKinley Meyer]{
\begin{itemize}
\item[-]
Department of Mathematics, Statistics and Computer Science, University of Illinois at Chicago, 851 S Morgan St, 322
SEO, Chicago, 60607, IL, USA;
\end{itemize}
} \email[McKinley Meyer]{mmeyer31@uic.edu}

\renewcommand{\thefootnote}{\fnsymbol{footnote}}
\renewcommand{\thefootnote}{\arabic{footnote}}

\renewcommand{\thefootnote}{\fnsymbol{footnote}}
\renewcommand{\thefootnote}{\arabic{footnote}}

\thanks{A.C.C. was partially supported  by  a Collaboration Grant for Mathematicians from the Simons Foundation  
under Award No. 709008.}


\begin{abstract}
Let $E$ be an elliptic curve defined over $\Q$ and without complex multiplication. 
For a prime $p$ of good reduction for $E$, we write $\#E_p(\F_p) = p + 1 - a_p(E)$ for the number
of $\F_p$-rational points of the reduction $E_p$ of $E$ modulo $p$.
Under the Generalized Riemann Hypothesis (GRH), we study the primes $p$ for which 
the integer $|a_p(E)|$ is a prime.
In particular,
we prove the following results:
(i) the number of primes $p < x$ for which $|a_p(E)|$ is a  prime
is bounded from above by $C_1(E) \frac{x}{(\log x)^2}$ for some  constant $C_1(E)$;
(ii) the number of primes $p < x$ for which $|a_p(E)|$ is the product of at most  4  distinct primes,
counted without multiplicity,
is bounded from below by $C_2(E) \frac{x}{(\log x)^2}$ for some  constant $C_2(E)$;
(iii) the number of primes $p < x$ for which $|a_p(E)|$ is the product of at most  5 distinct primes,
counted with multiplicity,
is bounded from below by $C_3(E) \frac{x}{(\log x)^2}$ for some positive constant $C_3(E) > 0$.
Under GRH, we also prove the convergence of the sum of the reciprocals of the primes $p$ for which $|a_p(E)|$ is a prime.
Furthermore, under GRH, together with Artin's Holomorphy Conjecture and a Pair Correlation Conjecture for Artin L-functions, 
we prove that
the number of primes $p < x$ for which $|a_p(E)|$ is the product of at most  2 distinct primes,
counted with multiplicity,
is bounded from below by $C_4(E) \frac{x}{(\log x)^2}$ for some  constant $C_4(E)$.
The constants $C_i(E)$, $1 \leq i \leq 4$, are defined explicitly  in terms of $E$ and are factors of another explicit constant $C(E)$ that appears 
in the conjecture that
$\#\{p < x: |a_p(E)| \ \text{is prime}\} \sim C(E) \frac{x}{(\log x)^2}$.
\end{abstract}

\maketitle


\section{Introduction}\label{sec_introduction}

In the mid 1600s, Fermat stated a characterization in terms of modular arithmetic
 for the primes that may be written as $m^2 + D n^2$ for some integers $m, n$,
for each  $D \in \{ 1, 2,  3\}$. His statements sparked the development of significant  branches of contemporary number theory,
which, in turn, provided the tools for proving  a characterization of the primes $p$ represented by a given arbitrary
positive definite, primitive, integral, binary quadratic form (see \cite{Co89} for a beautiful account of this proof and its history).
An outcome of this characterization is that, as $x \rightarrow \infty$,
\begin{equation}\label{primes-Q(m, n)}
\#\left\{
 p \leq x:   p = Q(m, n) \ \text{for some  integers} \ m, n  \right\}
 \sim
 \
C_0(Q) \frac{x}{\log x},
\end{equation}
where 
$C_0(Q)$ is  some explicit  positive constant that depends on the binary quadratic form $Q$.

 In 1997, Fouvry and Iwaniec \cite{FoIw97} pursued the study of  primes $p$ that may be written as
$m^2 + n^2$ for some 
 integers $m, n$ such that $m$ is a prime.
In particular, they proved the striking result that, as $x \rightarrow \infty$,
 \begin{equation}\label{counting-ell^2+n^2}
 \#\left\{
 p \leq x:   p = m^2 + n^2 \ \text{for some integers} \ m, n \ \text{such that} \ m \ \text{is prime}
 \right\}
 \sim
 \
C \frac{x}{(\log x)^2},
 \end{equation}
where $C$ is some explicit positive constant.
Moreover, in 2010, Friedlander and Iwaniec \cite[Thm. 18.6]{FrIw10}
provided a simplified proof of (\ref{counting-ell^2+n^2}),
while in 2020, Lam, Schindler, and Xiao \cite{LaScXi20} proved a generalization of (\ref{counting-ell^2+n^2})
which applies to any arbitrary
positive definite, primitive, integral, binary quadratic form
$Q(m, n)$, 
and not only
to the form $Q(m, n) = m^2 + n^2$. In particular, they proved that, 
as $x \rightarrow \infty$,
 \begin{equation}\label{counting-Q(ell, n)}
 \#\left\{
 p \leq x:   p = Q(m, n) \ \text{for some  integers} \ m, n \ \text{such that} \ m \ \text{is prime}
 \right\}
 \sim
 \
C(Q) \frac{x}{(\log x)^2},
 \end{equation}
 where $C(Q)$ is some explicit  constant that depends on the binary quadratic form $Q$.
 These results may be viewed under the unifying theme of equidistribution and primes envisioned by Sarnak \cite{Sa08}.
 
Primes of the forms $m^2 + n^2$ for some 
integers $m, n$ such that $m$ is a prime
also appear naturally in the setting of elliptic curves. Indeed,
for the elliptic curve $E: y^2 = x^3 - x$,
any odd prime $p$ for which the Frobenius trace $a_p(E)$ 
is the double of a prime $\ell$ gives rise to the representation 
$p = \ell^2 + n^2$ for some non-zero integer $n$
(see below for the definition of $a_p(E)$).
This is not an isolated example, but rather a particular case of
a general phenomenon related to a pair $(E, p)$, 
where
$E$
is an elliptic curve defined over $\Q$  and with complex multiplication,
and 
$p$ is a prime of good reduction for $E$ having the property that 
   the Frobenius trace $a_p(E)$ 
 is a prime or twice a prime.

With this motivation in mind, 
our purpose in this paper is to investigate 
the primality of the Frobenius traces of an elliptic curve defined over $\Q$ and without complex multiplication,
as explained in what follows.

Let $E$ be an elliptic curve defined over $\Q$, of conductor $N_E$, and
without complex multiplication.
For each rational prime $p \nmid N_E$, we denote by
$E_p$ the reduction of $E$ modulo $p$. We recall that the reduced curve $E_p$ is itself an elliptic curve over the finite field $\F_p$ with $p$ elements
and that its group of $\F_p$-rational points, $E_p(\F_p)$, has size $\#E_p(\F_p) = p + 1 - a_p(E)$ for some 
integer $a_p(E)$ which satisfies the bound 
\begin{equation}\label{weil-bound}
|a_p(E)| < 2 \sqrt{p}.
\end{equation}
For a positive integer $k$,
we denote by
$\omega(k)$ the number of prime factors of $k$, counted without multiplicity,
and by
$\Omega(k)$ the number of prime factors of $k$, counted with multiplicity.
A positive integer $n$ is called a {\it{$Q_k$-integer}} if
$\omega(n) \leq k$, and  a {\it{$P_k$-integer}} if $\Omega(n) \leq k$.
With this notation and terminology, 
for an arbitrary $x > 0$ and an arbitrary integer $k \geq 1$, we set
\begin{equation}\label{pi-twin-ec-x}
\pi_{E, \text{prime trace}}(x)
:=
\#
\{
p \leq x:
p \nmid N_E,
|a_p(E)|  \ \text{is prime}
\},
\end{equation}
\begin{equation}\label{pi-almost-twin-ec-x-omega}
\pi_{E, \text{$Q_k$-trace}}(x)
:=
\#
\{
p \leq x:
p \nmid N_E,
a_p(E) \not\in \{0, \pm 1\},
\omega(|a_p(E)|) \leq k
\},
\end{equation}
\begin{equation}\label{pi-almost-twin-ec-x-Omega}
\pi_{E, \text{$P_k$-trace}}(x)
:=
\#
\{
p \leq x:
p \nmid N_E,
a_p(E) \not\in \{0, \pm 1\},
\Omega(|a_p(E)|) \leq k
\}.
\end{equation}

\medskip

In \cite{CoJo22}, Cojocaru and Jones make the following prediction:

\medskip

{\bf{Conjecture}}
\\
{\emph{
\hspace*{0.5cm}
Let $E$ be an elliptic curve defined over $\Q$, of conductor $N_E$, and
without complex multiplication.
There exists a constant $C(E)$ such that, as $x \rightarrow \infty$,
\begin{equation}\label{conjecture-function}
\pi_{E, \text{prime trace}}(x)
\sim C(E) \frac{x}{(\log x)^2}.
\end{equation}
The constant is explicitly defined as 
\begin{equation}\label{conjecture-constant}
C(E) 
:= 
2
\cdot
\frac{m_E}{\phi(m_E)}
\cdot
\frac{
\#\left\{
M \in \Gal(\Q(E[m_E])/\Q): 
\tr M \in (\Z/m_E \Z)^{\times}
\right\}
}{
\#\Gal(\Q(E[m_E])/\Q)}
\cdot
\ds\prod_{\ell \nmid m_E \atop{\ell \ \text{prime}}} 
\left(
1 - \frac{1}{\ell^3 - \ell^2 - \ell + 1}
\right),
\end{equation}
where $m_E$ is the torsion conductor of $E/\Q$ (see Section \ref{sec_prelim_ec-div} for a definition),
$\Gal(\Q(E[m_E])/\Q)$ is the Galois group of the $m_E$-division field of $E$, viewed as a subgroup of 
the matrix group $\GL_2(\Z/m_E \Z)$,
and
$\phi(m_E)$ is the Euler function of $m_E$.
}}

\medskip

\noindent
Let us note that, 
in \cite{CoJo22}, the authors show that the constant $C(E)$ is positive for infinitely many elliptic curves $E$.

\medskip

A weaker form of this conjecture, with $C(E)$ not given explicitly, was proposed by Cojocaru as a topic of 
investigation in the Bachelor's thesis of Lane \cite{La05}. Therein, it was proven that, under the Generalized Riemann Hypothesis for Dedekind zeta functions, 
we have
\begin{equation}\label{lane-upper}
\pi_{E, \text{prime trace}}(x) \ll_E \frac{x}{(\log x)^2},
\end{equation}
\begin{equation}\label{lane-lower-omega}
\pi_{E, \text{$Q_5$-trace}}(x) \gg_E \frac{x}{(\log x)^2},
\end{equation}
and
\begin{equation}\label{lane-lower-Omega}
\pi_{E, \text{$P_7$-trace}}(x) \gg_E  \frac{x}{(\log x)^2},
\end{equation}
where the constants implied in the $\ll_E$ and $\gg_E$ notation depend on $E$ in an unspecified way.

\medskip

Our goal in the paper is to improve upon the above results in several aspects, as follows.

\begin{theorem}\label{main-thm-upper}
Let $E$ be an elliptic curve defined over $\Q$, of conductor $N_E$, and without complex multiplication.
Assume that there exists some $\frac{1}{2} \leq \theta < 1$ such that the $\theta$-quasi Generalized Riemann Hypothesis holds  for  Dedekind zeta functions. 
Then, for all $x > e$,
\begin{equation}\label{upper-bound-growth-theta}
\pi_{E, \text{prime trace}}(x) \leq \left(\frac{3}{1-\theta} + \o(1)\right)C(E) \frac{x}{(\log x)^2},
\end{equation}
where $C(E)$ is the explicit constant introduced in (\ref{conjecture-constant}).	
In particular, when $\theta = \frac{1}{2}$, (\ref{upper-bound-growth-theta}) becomes
\begin{equation}\label{upper-bound-growth-GRH}
\pi_{E, \text{prime trace}}(x)\leq \left(6 + \o(1)\right)C(E) \frac{x}{(\log x)^2}.
\end{equation}
\end{theorem}

\medskip

As a corollary to Theorem \ref{main-thm-upper}, we obtain the convergence of the sum of $1/p$, where $p$ runs over primes for which $a_p(E)$ is a prime,
a result reminiscent of a famous theorem of Brun 
about the convergence of the sum of $1/p$, where $p$ runs over primes for which $p+2$ is also a prime.

\begin{corollary} \label{Brun-analog}
	Let $E$ be an elliptic curve defined over $\Q$, of conductor $N_E$, and without complex multiplication.
	Assume that there exists some $\frac{1}{2} \leq \theta < 1$ such that the $\theta$-quasi Generalized Riemann Hypothesis holds  for  Dedekind zeta functions. 
	Then 
	$$\ds\sum_{\substack{{p \ \nmid N_E}\\{a_p(E) \text{ prime}}}} \frac{1}{p} < \infty.$$ 
	More precisely, for any $\e > 0$, there exists $x_0 = x_0(E,\theta,\e)$ such that
	\begin{equation*}
		\sum_{\substack{{p \geq x_0}\\{ a_p(E) \text{ prime}}}} \frac{1}{p} \leq \left(\frac{3}{1-\theta} + \e\right) C(E) \frac{1}{\log x_0},
	\end{equation*}
	where $C(E)$ is the explicit constant introduced in  (\ref{conjecture-constant}).
\end{corollary}

Our second main result  is reminiscent of Chen's lower bound \cite{Ch73} about  almost twin primes.

\begin{theorem}\label{main-thm-lower}
	Let $E$ be an elliptic curve defined over $\Q$, of conductor $N_E$, and without complex multiplication.
	\begin{enumerate}
	\item[](i)
	Assume that there exists some $\frac{1}{2} \leq \theta < 1$ such that the $\theta$-quasi Generalized Riemann Hypothesis holds  for  Dedekind zeta functions. 
	Denote by $r_1, r_2$ the positive integers defined by
	\begin{align*}
		&r_1 = r_1(\theta) := 1 + \left[ \frac{1}{0.83}\left(\frac{3}{2(1-\theta)} - \frac{1}{6} \right) \right], \\
		&r_2 = r_2(\theta) := 1 + \left[ \frac{5}{2(1-\theta)} - \frac{5}{12} \right].
	\end{align*}
	Then, for all $x > e$,
	\begin{equation}\label{lower-bound-growth-theta-omega}
		\pi_{E, \text{$Q_{r_1}$-trace}}(x)
		\geq
		\frac{3}{1-\theta}(0.00692 + \o(1))C(E) \frac{x}{(\log x)^2}
	\end{equation}
	and
	\begin{equation}\label{lower-bound-growth-theta-Omega}
		\pi_{E, \text{$P_{r_2}$-trace}}(x)
		\geq
		\frac{3}{1-\theta}(0.3162 + \o(1))C(E) \frac{x}{(\log x)^2}.
	\end{equation}
	In particular, when $\theta = \frac{1}{2}$, 
	the upper bounds
	(\ref{lower-bound-growth-theta-omega}) and (\ref{lower-bound-growth-theta-Omega})
	become
	\begin{equation*}\label{lower-bound-growth-GRH-omega}
		\pi_{E, \text{$Q_4$-trace}}(x)
		\geq
		(0.0415 +\o(1))C(E) \frac{x}{(\log x)^2}
	\end{equation*}
	and
	\begin{equation*}\label{lower-bound-growth-GRH-Omega}
		\pi_{E, \text{$P_5$-trace}}(x)
		\geq
		(1.8972 + \o(1))C(E) \frac{x}{(\log x)^2}.
	\end{equation*}
	\item[(ii)]
	Assume that the Generalized Riemann Hypothesis holds  for  Dedekind zeta functions, and that
	Artin's Holomorphy Conjecture and a Pair Correlation Conjecture hold  for Artin L-functions. 
	Then, for all $x > e$,
		\begin{equation}\label{lower-bound-growth-PCC-Omega}
		\pi_{E, \text{$P_{ 2  }$-trace}}(x)
		\geq
		(1 + \o(1))C(E) \frac{x}{(\log x)^2}.
	\end{equation}
	\end{enumerate}
\end{theorem}

\bigskip

Our approach to the study of the primality of $a_p(E)$ takes inspiration 
from classical studies of the primality of $p + a$, for some fixed even integer $a$.
In particular, it draws on the analogy between
 the interpretation of the divisibility
$m \mid a_p(E)$ for some non-zero integer $m$
as a Chebotarev condition at $p$ in the division field $\Q(E[m])$ of $E$
and 
the interpretation of the divisibility
$m \mid (p + a)$ for some non-zero integer $m$
as a Chebotarev condition at $p$ in the cyclotomic field $\Q(\zeta_m)$.
For example, in light of this analogy,
analytic methods, such as Tur\'{a}n's normal order method and Titchmarsh' conditional approach to the divisor problem 
for $p+a$,
 have already been used in the study of the divisors of $a_p(E)$ in 
 works such as \cite{GuMu14},  \cite{MuMu84}, and \cite{Po16}.
 
  In this paper, we pursue the  analogy between 
  the arithmetic of $a_p(E)$ and that of $p + a$
  by employing sieve methods
  in the spirit of the study of the primality of $p + 1 - a_p(E)$  in analogy to that of $p+a$
  pursued in 
  \cite{Co05}, 
  \cite{DaWu12},
  \cite{IwJU10},
  \cite{JU08},
  \cite{MiMu01},
  and 
  \cite{StWe05}.
  In particular, 
  after setting an elliptic curve sieve problem in 
  Section \ref{sec_sieve_setting},
  we use a version of the Selberg Upper Bound Sieve from \cite{HaRi74}
  to prove Theorem \ref{main-thm-upper} in 
  Section \ref{sec_proof_upper},
  and
  we use a version of Greaves Lower Bound Sieve from \cite{HaRi85}
  to prove Theorem \ref{main-thm-lower} in 
  Section \ref{sec_proof_lower}.
  The key ingredients in these proofs are
  Theorems \ref{corollary-sizes-cond-sets} and \ref{corollary-sizes-cond-sets-PCC}
  from 
  Section \ref{sec_prelim_ec}, 
  based on conditional  effective versions of the Chebotarev Density Theorem.
  The special feature of these density theorems is 
   their number field setting, which is   derived  not from the division field $\Q(E[m])$ of $E$,
  but rather from the splitting field $J_{E, m}$ of the modular polynomial $\Phi_m(X, j(E))$,
  where $j(E)$ is the $j$-invariant of the elliptic curve $E$
  (see
  (\ref{define-J_m})
  and
   (\ref{Second setting}) in 
   Section \ref{sec_prelim_ec}
  for this number field setting, 
  and 
  \cite[Ch. 3]{Co89} for modular polynomials).
  This choice of number field setting leads to  improvements from $Q_5$ to $Q_4$, 
  and from $P_7$ to $P_5$, 
  in
  Theorem \ref{main-thm-lower}
  over 
  (\ref{lane-lower-omega}) - (\ref{lane-lower-Omega}), 
  under the Generalized Riemann Hypothesis, 
  and to the $P_2$ result of 
  Theorem \ref{main-thm-lower}
  under the Generalized Riemann Hypothesis, 
 Artin's Holomorphy Conjecture, and a Pair Correlation Conjecture.
 The improvements in the explicit constants of $E$ emerging in all  of our main results
 come from the specific versions of upper bound and lower bound sieves that we use.

\bigskip
\noindent
{\bf{Notation}}. 
Throughout the paper, we use the following  notation.

\noindent
$\bullet$
Given a finite set $S$, we denote its cardinality by 
$\# S$.

\noindent
$\bullet$
Given suitably defined real functions $h_1, h_2$,
we say that 
$h_1 = \o(h_2)$ if $\ds\lim_{x \rightarrow \infty} \frac{h_1(x)}{h_2(x)} = 0$;
we say that
$h_1 = \O(h_2)$ or, equivalently, that $h_1 \ll h_2$, 
if 
$h_2$ is positive valued 
and
 there exists a positive constant $c$ such that 
$|h_1(x)| \leq c \ h_2(x)$ for all $x$ in the common domain of $h_1$ and $h_2$;
we say that $h_1 \asymp h_2$  if  $h_1$, $h_2$ are positive valued  and  $h_1 \ll h_2 \ll h_1$;
we say that
$h_1 = \O_D(h_2)$ or, equivalently, that $h_1 \ll_D h_2$,
if 
$h_1 = \O(h_2)$
and
the implied $\O$-constant $c$ depends on priorly given data  $D$;
we say that 
$h_1 \sim h_2$ 
if
$\ds\lim_{x \rightarrow \infty} \frac{h_1(x)}{h_2(x)} = 1$.

\noindent
$\bullet$
Given  integers $m \geq 2$ and $n \geq 2$, we write $m \mid n^{\infty}$ to mean that every prime dividing $m$ 
also divides $n$.

\noindent
$\bullet$
Given an integer $n \geq 1$,
we write 
$\phi(n)$ for the Euler function of $n$,
$\omega(n)$ for the number of prime factors of $n$, counted without multiplicity,
and
$\Omega(n)$ for the number of  prime factors of $n$, counted with multiplicity.

\noindent
$\bullet$
We use the letters $p$ and $\ell$ to denote positive rational primes.
We denote by $\pi(x)$ the number of primes $p \leq x$
and 
we recall that, by the Prime Number Theorem,
$\pi(x) \sim \frac{x}{\log x}$. 

\noindent
$\bullet$
Given an odd prime $\ell$,
we use the  notation
$\left(\frac{\cdot}{\ell}\right)$ for the Legendre symbol.

\noindent 
$\bullet$
 Given an integer $m \geq 1$, we denote by $\Z/m \Z$ the ring of integers modulo $m$.
 When $m$ is a prime $\ell$, we denote $\Z/\ell \Z$ by $\F_{\ell}$ to emphasize 
 its
 field structure.
For an integer $a$, we denote by $a (\mod \ell)$ its residue class modulo $\ell$.
 
\noindent 
$\bullet$
Given a prime $\ell$, we denote by $\Z_{\ell}$ the ring of $\ell$-adic integers.
We set $\hat{\Z} := \ds \lim_{\leftarrow \atop m} \Z/m \Z$
and recall that there exists a ring isomorphism 
$\hat{\Z}  \simeq \ds\prod_{\ell} \Z_{\ell}$.

\noindent
$\bullet$
 Given a number field $K$, 
 we denote  
 by  ${\cal{O}}_K$ its ring of integers,
 by ${\sum}_K$ the set of non-zero prime ideals of  ${\cal{O}}_K$,
 by $[K:\Q]$ the degree of $K$ over $\Q$,
 by $d_K \in \Z \backslash \{0\}$ the discriminant of an integral basis of ${\cal{O}}_K$,
 and 
 by  $\disc(K/\Q) = \Z d_K \unlhd \Z$ the discriminant ideal of $K/\Q$.
 For a prime ideal $\mathfrak{p} \in {\sum}_K$, 
 we denote by $N_{K/\Q}(\mathfrak{p})$ its norm in $K/\Q$.
We say that $K$ satisfies the Generalized Riemann Hypothesis (GRH) if
 the Dedekind zeta function $\zeta_K$ of $K$ has the property that,
 for any $\rho \in \C$ with $0 \leq \Re \rho \leq 1$ and $\zeta_K(\rho) = 0$, we have $\Re(\rho) = \frac{1}{2}$. 
For $\frac{1}{2} \leq \theta < 1$,
we say that $K$ satisfies the $\theta$-quasi Generalized Riemann Hypothesis ($\theta$-quasi GRH) if
 the Dedekind zeta function $\zeta_K$ of $K$ has the property that,
 for any $\rho \in \C$ with $0 \leq \Re \rho \leq 1$ and $\zeta_K(\rho) = 0$, we have $\Re(\rho) \geq \theta$.
 Note that $\frac{1}{2}$-quasi GRH is the same as GRH.

\noindent
$\bullet$
Given  a  Galois extension $L/K$ of number fields
and given an irreducible character $\chi$ of the Galois group of $L/K$,
we denote 
by $\mathfrak{f}(\chi) \unlhd {\cal{O}}_K$ the global Artin conductor of $\chi$,
by $A_{\chi} := |d_L|^{\chi(1)} N_{K/\Q}(\mathfrak{f}(\chi)) \in \Z$ the conductor of $\chi$,
and by ${\cal{A}}_{\chi}(T)$ the function of a positive real variable $T > 3$ defined by the relation
$$
\log {\cal{A}}_{\chi}(T) = \log A_{\chi} + \chi(1) [K:\Q] \log T.
$$
 
\noindent
$\bullet$
Given  a  Galois extension $L/K$ of number fields,
we say that it satisfies Artin's Holomorphy Conjecture (AHC)
if, for any irreducible character $\chi$ of the Galois group of $L/K$,
the Artin L-function $L(s, \chi, L/K)$ extends to a function that is
analytic on the whole $\C$, except at $s = 1$ when $\chi = 1$.
We recall that, if
we assume GRH for  the number field $L$ and AHC for the number field extension $L/K$, 
then, given 
any  irreducible character $\chi$ of the Galois group of $L/K$,
and given any non-trivial zero 
$\rho$ of $L(s, \chi, L/K)$, 
 the real part  of $\rho$ satisfies
$\Re \rho = \frac{1}{2}$.
In this case, we write $\rho = \frac{1}{2} + i \gamma$, where 
$\gamma$ denotes the imaginary part 
of $\rho$.

\noindent
$\bullet$
Given  a  Galois extension $L/K$ of number fields,
let us assume GRH for $L$ and AHC for $L/K$.
For an  irreducible character $\chi$ of the Galois group of $L/K$
and an arbitrary $T > 0$, 
we define the pair correlation function of $L(s, \chi, L/K)$ by
\begin{equation*}\label{pair-cor-fcn}
{\cal{P}}_T (X, \chi)
:=
\ds\sum_{- T \leq \gamma_1 \leq T}
\ds\sum_{- T \leq \gamma_2 \leq T}
w(\gamma_1 - \gamma_2)
e((\gamma_1 - \gamma_2) X),
\end{equation*}
where
$\gamma_1$ and $\gamma_2$ range over all the imaginary parts of the non-trivial zeroes
$\rho = \frac{1}{2} +  i \gamma$
 of $L(s, \chi, L/K)$,
counted with multiplicity, 
and where,
 for an arbitrary real number $u$,
$e(u) := \exp(2 \pi i u)$ and $w(u) := \frac{4}{4 + u^2}$.
We say that the extension $L/K$ satisfies the Pair Correlation Conjecture (PCC) 
if, 
 for any irreducible character $\chi$ of the Galois group of $L/K$ and for any $A > 0$ and $T > 3$,
 provided
$0 \leq Y \leq A   \chi(1) [K:\Q] \log T$,
we have
$$
{\cal{P}}_T(Y, \chi) \ll_A \chi(1)^{-1} T \log {\cal{A}}_{\chi}(T).
$$

\noindent
$\bullet$
Given a field $F$, we denote 
by  $\overline{F}$ a fixed algebraic closure of $F$.
 
\noindent
$\bullet$
Given a non-zero unitary commutative ring $R$, we denote by
$R^{\times}$  its group of multiplicative units.

\noindent 
$\bullet$
Given a non-zero unitary commutative ring $R$ and an integer $n \geq 1$, 
we denote by
$M_{n}(R)$
the ring of $n \times n$ matrices with entries in $R$
and by $I_n$ the  identity  matrix in $M_n(R)$.
For an arbitrary matrix $M \in M_n(R)$,
we denote by $\tr M$ and $\det M$ its trace and determinant,
and by $M^t$ its transpose.
We define the general linear group  $\GL_{n}(R)$
 as the collection of  $M \in M_n(R)$ with $\det M \in R^{\times}$.
 We define the projective general linear group  $\PGL_{n}(R)$
 as the quotient group $\GL_{n}(R)/\left\{a I_n: a \in R^{\times}\right\}$.

\medskip

\noindent
{\bf{Acknowledgements.}}
We thank 
Damaris Schindler for useful remarks related to the asymptotic formula (\ref{counting-Q(ell, n)})
and
Nathan Jones for helpful conversations related to the constants
(\ref{conjecture-constant}) and (\ref{conjecture-constant-bis}).

\section{Preliminaries about elliptic curves}\label{sec_prelim_ec}

\subsection{Generalities about  Galois representations associated to elliptic curves}\label{sec_prelim_ec-repres}

Let $E$ be an elliptic curve defined over $\Q$ and  of conductor $N_E$. For a fixed arbitrary positive integer $m$, we denote
by $E[m]$ the group of $\overline{\Q}$-rational points of $E$ of order dividing $m$
and by $\Q(E[m])$ the field obtained by adjoining  to $\Q$ the $x$ and $y$ coordinates of
the points of $E[m]$.  
We recall from the theory of elliptic curves that 
the group  $E[m]$ is isomorphic to $(\Z/m \Z)^2$,
that
the field extension $\Q(E[m])/\Q$ is finite and Galois, 
and 
that 
the rational primes that ramify in $\Q(E[m])$ are among the prime factors of $m N_E$.
By fixing a $\Z/m\Z$-basis of $E[m]$, we obtain 
{\it{the residual modulo $m$  Galois representation of $E/\Q$,}}
$$
\overline{\rho}_{E, m}: \Gal(\overline{\Q}/\Q) \longrightarrow \GL_2(\Z/m \Z),
$$
which has the property that
\begin{equation}\label{ker-rho-m}
\Q(E[m]) = \overline{\Q}^{\Ker \overline{\rho}_{E, m}}.
\end{equation}
Taking the inverse limit over $m$ of the representations $\overline{\rho}_{E, m}$, we obtain
a continuous Galois representation,
{\it{the adelic  Galois representation of $E/\Q$,}}
$$
\rho_{E}: \Gal(\overline{\Q}/\Q) \longrightarrow \GL_2\left(\hat{\Z}\right).
$$
Setting $m$ to be powers $\ell^k$ of a fixed prime $\ell$ and taking the inverse limit over $k$
of the representations $\overline{\rho}_{E, \ell^k}$, we obtain
another continuous representation,
{\it{the $\ell$-adic  Galois representation of $E/\Q$,}}
$$
\rho_{E, \ell}: \Gal(\overline{\Q}/\Q) \longrightarrow \GL_2\left(\Z_{\ell}\right).
$$

It is known that, for each prime $p \nmid m N_E$, the $p$-Weil polynomial
$$
P_{E, p}(X) := X^2 - a_p(E) X + p \in \Z[X]
$$ 
satisfies the congruence
$$
P_{E, p}(X) 
\equiv 
\det\left (
X I_2 - \overline{\rho}_{E, m}\left(
\Frob_{\mathfrak{p}}
\right)
\right)(\mod m),
$$
where
$
\Frob_{\mathfrak{p}}
 \in \Gal(\Q(E[m])/\Q)$ 
denotes the Frobenius element at an arbitrary non-zero prime ideal $\mathfrak{p}$ of
$\Q(E[m])$, lying above $p$.
Thus, we always have the congruence
\begin{equation}\label{trace-mod-m}
\tr \overline{\rho}_{E, m}\left( 
\Frob_{\mathfrak{p}}
\right) 
\equiv a_p(E) \: (\mod m).
\end{equation}
This congruence 
suggests that
the field extension $\Q(E[m])/\Q$ plays a crucial role in the study of the arithmetic properties of $a_p(E)$,
an observation  which we will put to use in our proofs.

\subsection{Division fields of elliptic curves}\label{sec_prelim_ec-div}

As in Subsection \S \ref{sec_prelim_ec-repres},
let $E$ be an elliptic curve defined over $\Q$ and of conductor $N_E$,
and
let $m$ be an arbitrary positive integer. 
Thanks to (\ref{ker-rho-m}), the Galois group 
$\Gal(\Q(E[m])/\Q)$, which we  denote by
$$ G_E(m) := \Gal(\Q(E[m])/\Q),$$
may be identified with a subgroup
of $\GL_2(\Z/m \Z)$.
As a consequence, the degree of the extension $\Q(E[m])/\Q$ has the natural upper bound
\begin{equation}\label{m-division-upper}
\left[\Q(E[m]):\Q\right] 
\leq 
\#\GL_{2}(\Z/m \Z)
= 
m^4 \ds\prod_{\ell \mid m} \left(1 - \frac{1}{\ell}\right) \left(1 - \frac{1}{\ell^2}\right)
\leq 
m^4.
\end{equation}

If $E/\Q$ is without complex multiplication, then Serre's Open Image Theorem for elliptic curves, 
proven in \cite{Se72}, implies the existence of
a smallest positive integer $m_E$ having the property that,
upon writing the priorly fixed arbitrary integer $m$ uniquely as 
\begin{equation}\label{m-factor}
m = m_1 m_2
\end{equation}
for some positive integers $m_1, m_2$ such that
\begin{equation*}\label{m_1-m_2}
m_1 \mid m_E^{\infty} \ \ \text{and} \ \ \gcd(m_2, m_E) = 1,
\end{equation*}
there exists a subgroup $H_{E, m_1} \leq \GL_2(\Z/m_1 \Z)$
satisfying a group isomorphism
\begin{equation}\label{serre-open-m-division}
G_E(m) \simeq H_{E, m_1} \times \GL_{2}(\Z/m_2 \Z).
\end{equation}
For future purposes, we recall that $m_E$ is an even positive integer (see \cite{Jo10}),
which we  refer to as {\it{the torsion conductor of $E/\Q$. }}

As a consequence of (\ref{serre-open-m-division}), 
if $E/\Q$ is without complex multiplication, then
 the degree of $\Q(E[m])/\Q$ 
 is the product of the  function of $m_1$ defined by $[H_{E, m_1} :\Q]$
 and 
  the explicit multiplicative function of $m_2$  defined by  $\#\GL_{2}(\Z/m_2 \Z)$.
Consequently,  the degree of $\Q(E[m])/\Q$ obeys the lower bound
\begin{equation*}\label{m-division-lower-one}
m_2^4 \ds\prod_{\ell \mid m_2} \left(1 - \frac{1}{\ell}\right)  \left(1 - \frac{1}{\ell^2}\right)
=
\#\GL_2(\Z/m_2 \Z)
\leq
\left[\Q(E[m]):\Q\right].
\end{equation*}

Our approach to the study of the prime factors of $a_p(E)$ will rely mostly on the properties of a particular 
subfield of the
division field $\Q(E[m])$, defined as follows.
Upon identifying $G_E(m)$ with its image
under $\overline{\rho}_{E, m}$
in $\GL_2(\Z/m \Z)$, we set
$J_{E, m}$ to be the subfield of $\Q(E[m])$ fixed by the scalar subgroup
$\Scal_{G_E(m)}$
 of $G_E(m)$,
 that is,
 \begin{equation}\label{define-J_m}
 J_{E, m} := \Q(E[m])^{\ds\Scal_{G_E(m)}},
 \end{equation}
 where
 \begin{equation*}\label{def-scalars}
 \Scal_{G_E(m)}
 :=
 G_E(m)
 \cap
 \left\{
 \begin{pmatrix} 
  a  & 0 \\ 0 & a  
  \end{pmatrix}
  \in \GL_2(\Z/m \Z) : 
  a \in (\Z/m \Z)^{\times}
  \right\}.
 \end{equation*}
 The subfield
 $J_{E, m}$ is, in fact, the splitting field
of the modular polynomial $\Phi_m(X, j(E))$ over $\Q$,
 but we will not use this property in our proofs.
  
 We observe that 
$ \Scal_{G_E(m)}
 \unlhd
 \: G_E(m)$
and deduce that
 $J_{E, m}/\Q$ is a finite Galois extension, whose Galois group we denote by
\begin{equation*}
	\widehat{G}_E(m) := \Gal(J_{E, m}/\Q).
\end{equation*}
 Moreover, we observe that
$J_{E, m} = \overline{\Q}^{\Ker \widehat{\rho}_{E, m}}$,
where
$$
\widehat{\rho}_{E, m}: \Gal(\overline{\Q}/\Q) \longrightarrow \PGL_2(\Z/m\Z)
$$ 
is the Galois representation obtained by composing
the natural projection
$\GL_2(\Z/m \Z) \to \PGL_2(\Z/m \Z)$
with
$\overline{\rho}_{E, m}$.
As a consequence, we obtain that the degree of $J_{E, m}/\Q$ satisfies the  upper bounds
\begin{equation}\label{m-j-division-upper}
[J_{E, m}:\Q] 
\leq 
\#\PGL_{2}(\Z/m \Z) 
=
m^3 \ds\prod_{\ell \mid m} \left(1 - \frac{1}{\ell^2}\right)
\leq 
m^3.
\end{equation}

If $E/\Q$ is without complex multiplication,
then, 
using  factorization (\ref{m-factor}) of $m$ and invoking Serre's Open Image Theorem as before,
we deduce that
\begin{equation}\label{serre-open-m-j}
\widehat{G}_E(m)
\simeq
\frac{G_E(m)}{\ds\Scal_{G_E(m)}}
\simeq 
\frac{H_{E, m_1}}{\ds\Scal_{H_{E, m_1}}} \times \PGL_{2}(\Z/m_2 \Z).
\end{equation}
Consequently, 
 the degree of $J_{E, m}/\Q$ 
 is the product of  the  function of $m_1$ defined by 
 $\frac{ \#H_{E, m_1} }{ \#\ds\Scal_{H_{E, m_1}} }$ 
 and the explicit multiplicative function of $m_2$  defined by  $\#\PGL_{2}(\Z/m_2 \Z)$.

\subsection{Applications of the Chebotarev Density Theorem for division fields of elliptic curves}\label{sec_ec_prelim-Cheb}

As in  Subsections \S \ref{sec_prelim_ec-repres} and \S \ref{sec_prelim_ec-div},
let $E$ be an elliptic curve defined over $\Q$, of conductor $N_E$,
and
let $m$ be an arbitrary positive integer. 
Throughout this subsection, we assume  that $E/\Q$ is without complex multiplication and use the
notation $m_E$ for the torsion conductor of $E/\Q$,
the even integer whose existence is ensured by Serre's Open Image Theorem for elliptic curves,
as mentioned in Subsection \S \ref{sec_prelim_ec-div}. 
Similarly to the previous two subsections, we use  factorization (\ref{m-factor}) for $m$
and we appeal to the group isomorphism (\ref{serre-open-m-division}) whenever needed.

Crucial to our analytic study of the primality of the Frobenius traces $a_p(E)$ of $E$ are applications in the settings
$\Q(E[m])/\Q$ and $J_{E, m}/\Q$  of an effective version of the Chebotarev Density Theorem, which we now recall.

Let $L/K$ be a  Galois extension of number fields, with $G:= \Gal(L/K)$, and let
$\emptyset \neq {\cal{C}} \subseteq G$ be a union of conjugacy classes of $G$.
 We denote 
  by $\disc(L/K) \unlhd {\cal{O}}_K$ the discriminant ideal of $L/K$.
We set
$$
\pi_{\cal{C}}(x, L/K)
:=
\ds\sum_{
\mathfrak{p} \in \Sigma_K
\atop{
\mathfrak{p} \nmid \disc(L/K)
\atop{
N_{K/\Q}(\mathfrak{p})\leq x
}
}
}
\delta_{\cal{C}}\left(
\left(\frac{L/K}{\mathfrak{p}}\right)
\right),
$$ 
where $\delta_{{\cal{C}}}(\cdot)$  is the characteristic function of ${\cal{C}}$,
 the sum is over non-zero prime ideals $\mathfrak{p}$ of ${\cal{O}}_K$ which are unramified in $L/K$ and
 have norm $N_{K/\Q}(\mathfrak{p})\leq x$,
 and
 $\left(\frac{L/K}{\mathfrak{p}}\right) \subseteq G$
 is the Artin symbol at $\mathfrak{p}$ in $L/K$.
 
  The Chebotarev Density Theorem asserts that, as $x \to \infty$,
 \begin{equation}\label{cheb-no-error}
 \pi_{\cal{C}}(x, L/K) 
  \sim \frac{\#{\cal{C}}}{\#G} \pi(x).
 \end{equation}
  In studies such as ours, the above asymptotic formula is needed in a formulation which highlights the dependence of the growth of the error term  $\left|\pi_{\cal{C}}(x, L/K) - \frac{\#{\cal{C}}}{\#G} \pi(x) \right|$
 on the extension $L/K$ and on the set ${\cal{C}}$.
 
 In order to state effective versions of (\ref{cheb-no-error}),  we introduce the notation
 $$
 P(L/K) := \{p: \: \ \exists \ \mathfrak{p} \in \Sigma_K
 \ \text{such that} \
 \mathfrak{p} \mid p \ \text{and} \ \mathfrak{p} \mid \disc(L/K)\}
 $$
 and
 $$
 M(L/K) := 2 [L:K] |d_K|^{\frac{1}{[K:\Q]}} \ds\prod_{p \in P(L/K)} p,
 $$
 and
 we recall
 from \cite[Prop. 5, p.~129]{Se81}
  that
\begin{equation}\label{hensel}
\log \left|N_{K/\Q} (\disc (L/K))\right|
\leq
([L : \Q] - [K : \Q])
\left(
\ds\sum_{p \in P(L/K)} \log p
\right)
+
[L : \Q] \log [L : K].
\end{equation}

We are now ready to state the effective versions of   (\ref{cheb-no-error}) needed in the proofs of our main results.

\begin{theorem}\label{cheb-LaOd}
Let $L/K$ be a Galois extension of number fields, with $G:= \Gal(L/K)$, and let
$\emptyset \neq {\cal{C}} \subseteq G$ be a union of conjugacy classes of $G$.
Assume that,
for some $\frac{1}{2} \leq \theta < 1$,
 the $\theta$-quasi-GRH holds for the Dedekind zeta function of  $L$.
 Then
\begin{equation*}
\pi_{\cal{C}}(x, L/K) 
= 
\frac{\#\cal{C}}{[L:K]} \pi(x)
+
\O\left(
(\#\cal{C}) \   x^{\theta} \ [K:\Q]
\left(
\frac{\log |d_L|}{[L:\Q]}
+
 \log x
\right)
\right).
\end{equation*}
 \end{theorem}
 \begin{proof}
The original reference is  \cite{LaOd77}. For the above formulation  with $\theta= \frac{1}{2}$,  see \cite[Thm. 4, p. 133]{Se81}. 
The same proof goes through for $\frac{1}{2} < \theta < 1$. 
 \end{proof}
 
 \begin{theorem}\label{cheb-MuMuWo}
Let $L/K$ be a Galois extension of number fields, with $G:= \Gal(L/K)$, and let
$\emptyset \neq {\cal{C}} \subseteq G$ be a union of conjugacy classes of $G$.
Denote by $\Gal(L/K)^{\#}$  the set of conjugacy classes of $\Gal(L/K)$.
Assume that GRH holds for the Dedekind zeta function of $L$, and that AHC and PCC hold for the extension $L/K$.
Then
 \begin{equation*}
 \pi_{\cal{C}}(x, L/K)
=
\frac{\#\cal{C}}{[L:K]} \pi(x)
+
\O\left(
(\#\cal{C})^{\frac{1}{2}} \left(\frac{\#\Gal(L/K)^{\#}}{[L:K]} \right)^{\frac{1}{2}}  x^{\frac{1}{2}} \ [K:\Q]^{\frac{1}{2}} 
\log (M(L/K)x)\right).
\end{equation*}
 \end{theorem}
 \begin{proof}
This  is \cite[Theorem 1.2, p. 402]{MuMuWo18}.
 \end{proof}

We now highlight  two particular elliptic curve settings for Theorems \ref{cheb-LaOd} - \ref{cheb-MuMuWo},
 which will be relevant to our study.
\begin{equation}\label{First setting}\tag{{\bf{First setting}}}
L= \Q(E[m])), \: K = \Q, \:\: {\cal{C}} = \mathcal{C}_{E}(m, \alpha)
\end{equation}
for a fixed $\alpha \in \Z$,
where
\begin{equation*}\label{class-m-without-hat}
\cal{C}_E(m, \alpha) := \left\{M \in G_E(m): \tr M = \alpha \: (\mod m) \right\};
\end{equation*}

\medskip

\begin{equation}\label{Second setting}\tag{{\bf{Second setting}}}
\hspace*{-0.8cm}
L= J_{E, m}, \: K = \Q, \: \: {\cal{C}} = \widehat{\cal{C}}_E(m, 0),
\end{equation}
where
\begin{equation*}\label{class-m-with-hat}
\widehat{\cal{C}}_E(m, 0) := \left\{\widehat{M} \in \widehat{G}_E(m): \tr M = 0 \: (\mod m)\right\},
\end{equation*}
with $M \in \GL_2(\Z/m \Z)$ denoting an arbitrary representative of a given coset $\widehat{M} \in \PGL_2(\Z/m \Z)$.

Observe that the group isomorphism
(\ref{serre-open-m-division})
 gives rise to the bijection
 \begin{eqnarray}\label{bijection-classes-without-hat}
\cal{C}_E(m, \alpha)
&\rightarrow&
\cal{C}_E(m_1, \alpha) \times \cal{C}(m_2, \alpha)
\\
M &\mapsto& \left(M_1, M_2\right),
\nonumber
\end{eqnarray}
where
\begin{equation*}\label{class-m1-without-hat}
\cal{C}_E(m_1, \alpha) 
:=
\left\{
M_1 \in H_{E, m_1}:
\tr M_1 = \alpha \: (\mod m_1)
\right\},
\end{equation*}
\begin{equation*}\label{class-m2-without-hat}
\cal{C}(m_2, \alpha)
:=
\left\{
M_2 \in \GL_2(\Z/m_2 \Z): \tr M_2 = \alpha \: (\mod m_2)
\right\}.
\end{equation*}
 Similarly, observe
 that the group isomorphism
 (\ref{serre-open-m-j}) 
 gives rise to the bijection
\begin{eqnarray}\label{bijection-classes-with-hat}
\widehat{\cal{C}}_E(m, 0)
&\rightarrow&
\widehat{\cal{C}}_E(m_1, 0) \times \widehat{\cal{C}}(m_2, 0)
\\
\widehat{M} &\mapsto& \left(\widehat{M}_1, \widehat{M}_2\right),
\nonumber
\end{eqnarray}
where
\begin{equation*}\label{class-m1-with-hat}
\widehat{\cal{C}}_E(m_1, 0) 
:=
\left\{
\widehat{M}_1 \in H_{E, m_1}/\Scal_{H_{E, m_1}}:
\tr M_1 = 0 \: (\mod m_1)
\right\},
\end{equation*}
\begin{equation*}\label{class-m2-with-hat}
\widehat{\cal{C}}(m_2, 0)
:=
\left\{
\widehat{M}_2 \in \PGL_2(\Z/m_2 \Z): \tr M_2 = 0 \: (\mod m_2)
\right\},
\end{equation*}
with
$M_1 \in H_{E, m_1}$ 
 an arbitrary representative of a given coset $\widehat{M}_1 \in H_{E, m_1}/\Scal_{H_{E, m_1}}$
and
with
$M_2 \in \GL_2(\Z/m_2 \Z)$ 
 an arbitrary representative of a given coset $\widehat{M}_2 \in \PGL_2(\Z/m_2 \Z)$.
 
 Using the above two number field settings and the above observations,
  in the next two propositions 
 we obtain immediate applications of Theorems \ref{cheb-LaOd} - \ref{cheb-MuMuWo}.

\begin{proposition}\label{cheb-LaOd-Jm}
Let $E$ be an elliptic curve defined over $\Q$, of conductor $N_E$,
 without complex multiplication, 
 and of torsion conductor $m_E$.
Let $m = m_1 m_2$ be a positive integer such that $m_1 \mid m_E^{\infty}$ and $\gcd(m_2, m_E) = 1$.
\begin{enumerate}
\item[(i)]
Let $\alpha \in \Z$.
Assume that,
for some $\frac{1}{2} \leq \theta < 1$,
the $\theta$-quasi-GRH holds for the Dedekind zeta function of $\Q(E[m])$.
Then
\begin{eqnarray*}
\#\left\{p \leq x: p \nmid m N_E,  a_p(E) \equiv \alpha \: (\mod m)\right\}
&=&
\
\frac{
\#\cal{C}_E(m_1, \alpha) 
\cdot
\#\cal{C}(m_2, \alpha)  
}{
\#H_{E, m_1}
\cdot
\#\GL_2(\Z/m_2 \Z)
}
\
 \pi(x)
 \\
&&
+
\O_E\left(
\#\cal{C}(m_2, \alpha) 
\ x^{\theta} \log (m x)\right).
\nonumber
\end{eqnarray*}
\item[(ii)]
Assume that,
for some $\frac{1}{2} \leq \theta < 1$,
the $\theta$-quasi-GRH holds for the Dedekind zeta function of $J_{E, m}$.
Then
\begin{eqnarray*}
\#\left\{p \leq x: p \nmid m N_E,  a_p(E) \equiv 0 \: (\mod m)\right\}
&=&
\
\frac{
\#\widehat{\cal{C}}_E(m_1, 0) 
\cdot
\#\Scal_{H_{E, m_1}}
\cdot \
\#\widehat{\cal{C}}(m_2, 0) 
}{
\#H_{E, m_1}
\cdot
\#\PGL_2(\Z/m_2 \Z)
}
\
 \pi(x)
 \\
&&
+
\O_E\left(
\#\widehat{\cal{C}}(m_2, 0) 
\ x^{\theta} \log (m x)\right).
\nonumber
\end{eqnarray*}
\end{enumerate}
\end{proposition}

\begin{proof}
Recalling  (\ref{hensel}) and that the ramified primes of $\Q(E[m])/\Q$, hence of $J_{E, m}/\Q$,
are among the prime factors of $m N_E$,
by applying 
(\ref{m-division-upper}),
respectively
(\ref{m-j-division-upper}), 
we deduce that
$$
\frac{\log |d_{\Q(E[m])}|}{[\Q(E[m]) : \Q]}
\leq
\ds\sum_{p \in {\cal{P}}(\Q(E[m])/\Q)} \log p
+
\log [\Q(E[m]) : \Q]
\ll 
\log (m N_E)
$$
and
$$
\frac{\log |d_{J_{E, m}}|}{[J_{E, m} : \Q]}
\leq
\ds\sum_{p \in {\cal{P}}(J_m/\Q)} \log p
+
\log [J_{E, m} : \Q]
\ll 
\log (m N_E).
$$
The asymptotic formulae claimed in the statement of the theorem now follow
from Theorem \ref{cheb-LaOd} by using
these estimates, along with 
(\ref{bijection-classes-without-hat})
and
(\ref{bijection-classes-with-hat}).
\end{proof}

\begin{proposition}\label{cheb-MuMuWo-Jm}
Let $E$ be an elliptic curve defined over $\Q$, of conductor $N_E$,
 without complex multiplication, 
 and of torsion conductor $m_E$.
Let $m = m_1 m_2$ be a positive integer such that $m_1 \mid m_E^{\infty}$ and $\gcd(m_2, m_E) = 1$.
Assume that
GRH holds for the Dedekind zeta function of $J_{E, m}$,
and
that AHC and PCC hold for the extension $J_{E, m}/\Q$.
Then
\begin{eqnarray*}
\#\left\{p \leq x: p \nmid m N_E,  a_p(E) \equiv 0 \: (\mod m)\right\}
&=&
\
\frac{
\#\widehat{\cal{C}}_E(m_1, 0) 
\cdot
\#\Scal_{H_{E, m_1}}
\cdot \
\#\widehat{\cal{C}}(m_2, 0) 
}{
\#H_{E, m_1}
\cdot
\#\PGL_2(\Z/m_2 \Z)
}
\
 \pi(x)
 \\
&&
+
\O_E\left(
\left(\#\widehat{\cal{C}}(m_2, 0)\right)^{\frac{1}{2}} 
\left(\frac{\#\PGL_2(\Z/m_2 \Z)^{\#}}{\#\PGL_2(\Z/m_2 \Z)} \right)^{\frac{1}{2}}
\ x^{\frac{1}{2}} 
\log (m  x)\right).
\nonumber
\end{eqnarray*}
\end{proposition}

\begin{proof}
The proof is similar to that of part (ii) of Proposition \ref{cheb-LaOd-Jm},
the only difference being that of applying 
Theorem \ref{cheb-MuMuWo} instead of Theorem \ref{cheb-LaOd}.
\end{proof}

To put these propositions to use, we need  a satisfactory understanding of the 
size of the unions of conjugacy classes 
that occur in  the main term and the error term 
 of these three asymptotic formulae. 
We record such counts below.

\begin{lemma}\label{class-tr-0-sizes}
Let $\ell$ be an odd prime and let $\alpha \in \Z$. 
Then	
        \begin{equation}\label{matrix-count-mod-ell-without-hat}
		\#\cal{C}(\ell,\alpha) = 
		\begin{cases}
		\ell^3 - \ell^2 - \ell & \text{ if } \alpha \not\equiv 0 \: (\mod \ell), \\
		\ell^3 - \ell^2 & \text{ if } \alpha \equiv 0 \: (\mod \ell);
		\end{cases}
	\end{equation}
	\begin{equation}\label{matrix-count-mod-ell-squared-without-hat}
		\#\cal{C}\left(\ell^2,\alpha \right) = \ell^6 - \ell^5 \ \text{ if } \alpha \equiv 0 \: (\mod \ell);
	\end{equation}
	\begin{equation}\label{matrix-count-mod-ell-with-hat}
		\#\widehat{\cal{C}}(\ell,0) = \ell^2;
	\end{equation}
	\begin{equation}\label{matrix-count-mod-ell-squared-with-hat}
		\#\widehat{\cal{C}}\left(\ell^2,0\right) = \ell^4.
	\end{equation}
\end{lemma}
\begin{proof}
Let us focus on proving formula (\ref{matrix-count-mod-ell-without-hat}) for $\#\cal{C}(\ell,\alpha)$ in the case $\alpha \equiv 0 \: (\mod \ell)$.
	First, we see easily that there are $\ell^3$ matrices with zero trace in $M_2(\Z/\ell\Z)$.
	Next, we determine how many of these matrices have zero determinant.
	Any matrix $M \in M_2(\Z/\ell\Z)$ with zero trace  can be written in the form
	\begin{equation*}
		M = \begin{pmatrix} a (\mod \ell)  & b (\mod \ell) \\ c (\mod \ell) & -a (\mod \ell)  \end{pmatrix}
	\end{equation*}
	for some cosets $a (\mod \ell), b (\mod \ell), c (\mod \ell) \in \Z/\ell \Z$.
	For any fixed pair $(b (\mod \ell), c (\mod \ell)) \in (\Z/\ell\Z)^2$, 
	there are $1 + (\frac{-bc}{\ell})$ possible cosets $a (\mod \ell)$ such that $\det M = 0 \: (\mod \ell)$.
Thus, the number of matrices $M \in M_2(\Z/\ell\Z)$ with $\tr M = \det M = 0 \: (\mod \ell)$ 
equals
	\begin{equation*}
		\ds\sum_{b (\mod \ell), c (\mod \ell) \in \Z/\ell\Z}
		 \left(1 + \left(\frac{-bc}{\ell}\right)\right)
	         = 
	         \ell^2 
	         + 
	         \ds\sum_{b (\mod \ell), c (\mod \ell) \in (\Z/\ell\Z)^{\times}}
	        	         \left(\frac{-bc}{\ell}\right)
		= 
		\ell^2.
	\end{equation*}
	We deduce that
	\begin{equation}\label{trace-zero-matrix-count-GL_2-mod-ell}
	\#\{M \in \GL_2(\Z/\ell\Z) : \tr M = 0 \: (\mod \ell) \}
	= \ell^3 - \ell^2 = \ell^2 \phi(\ell),
	\end{equation}
	establishing formula (\ref{matrix-count-mod-ell-without-hat}) for $\#\cal{C}(\ell,\alpha)$ in the case $\alpha \equiv 0 \: (\mod \ell)$.

	Now, let us focus on proving formula (\ref{matrix-count-mod-ell-without-hat}) for $\#\cal{C}(\ell,\alpha)$ in the case $\alpha \not\equiv 0 \: (\mod \ell)$.
	Note that, for any $\alpha_1 (\mod \ell), \alpha_2 (\mod \ell) \in (\Z/\ell\Z)^{\times}$, 
	by choosing $\beta \in \Z$ such that $\beta \equiv \alpha_2 \alpha_1^{-1} (\mod \ell)$, the map
	\begin{eqnarray*}
	\GL_2(\Z/\ell\Z) &\longrightarrow& \GL_2(\Z/\ell\Z)
	\\
	 M &\mapsto& \beta (\mod \ell)  \cdot M
	\end{eqnarray*}
	induces a bijection 
	$$\{ M \in \GL_2(\Z/\ell\Z) : \tr M = \alpha_1 (\mod \ell)\} \longrightarrow 
	\{ M \in \GL_2(\Z/\ell\Z) : \tr M = \alpha_2 (\mod \ell) \}.$$
	 This observation leads to  formula (\ref{matrix-count-mod-ell-without-hat})
	  for $\#\cal{C}(\ell,\alpha)$ in the case $\alpha \not\equiv 0 \: (\mod \ell)$.
	
	Next, observe that we may write any $M \in M_2(\Z/\ell^2\Z)$ 
	with $\tr M \equiv 0 \: (\mod \ell)$ uniquely in the form
	\begin{equation*}
		M = \begin{pmatrix} \hat{a}_0 +\hat{ a}_1 \ell  & \hat{b}_0 + \hat{b}_1 \ell \\ \hat{c}_0 + \hat{c}_1 \ell & -\hat{a}_0 - \hat{a}_1 \ell  \end{pmatrix}
	\end{equation*}
	for some $\hat{a}_0, \hat{a}_1, \hat{b}_0, \hat{b}_1, \hat{c}_0, \hat{c}_1  \in \Z/\ell\Z$. From here, the calculation is identical to the one for
	formula (\ref{matrix-count-mod-ell-without-hat})
	  for $\#\cal{C}(\ell, 0)$, except for taking into account that we have three free variables 
	  $\hat{a}_1, \hat{b}_1, \hat{c}_1$, so both the number of matrices with zero trace and the number of matrices with zero trace and  zero determinant  increase by a factor of $\ell^3$. 
	We deduce that
	\begin{equation}\label{trace-zero-matrix-count-GL_2-mod-ell-squared}
	\#\left\{M \in \GL_2\left(\Z/\ell^2\Z\right) : \tr M = 0 \: \left(\mod \ell^2\right) \right\}
	= \ell^6 - \ell^5 = \ell^4 \phi\left(\ell^2\right).
	\end{equation}

	From (\ref{trace-zero-matrix-count-GL_2-mod-ell}) and (\ref{trace-zero-matrix-count-GL_2-mod-ell-squared}), respectively,
	we conclude that
	\begin{equation*}
		\#\widehat{\cal{C}}(\ell,0) = \frac{\#\cal{C}(\ell,0) }{\phi(\ell)} = \ell^2
	\end{equation*}
	and
	\begin{equation*}
		\#\widehat{\cal{C}}(\ell^2,0) = \frac{\#\cal{C}(\ell^2,0) }{\phi(\ell^2)} = \ell^4,
	\end{equation*}
	which completes the proof of the lemma.
\end{proof}


Of primary interest to us are the following applications of part (ii) of Propositions \ref{cheb-LaOd-Jm} - \ref{cheb-MuMuWo-Jm}.

\begin{theorem}\label{corollary-sizes-cond-sets}
Let $E$ be an elliptic curve defined over $\Q$, of conductor $N_E$,
 without complex multiplication,  and of torsion conductor $m_E$.
\begin{enumerate}
\item[(i)]
Let $d$ be a squarefree positive integer such that $\gcd(d,  m_E) = 1$.
Assume that there exists  some $\frac{1}{2} \leq \theta < 1$
such that
the $\theta$-quasi-GRH holds for the Dedekind zeta function of $J_{E, d k}$
 for all positive squarefree integers $k$ with $k \mid m_E$.
Then
\begin{eqnarray*}
\#\left\{a_p(E): p \leq x, p \nmid N_E, \gcd(a_p(E), m_E) = 1, a_p(E) \equiv 0 \: (\mod d)\right\}
&=&
\frac{1}{d}
\left(
\ds\prod_{\ell \mid d}
\left( 1 - \frac{1}{\ell^2}\right)^{-1}
\right)
C_1(E)
 \pi(x)
 \\&&
+
\O_E\left(d^2 x^{\theta} \log (d x)\right),
\end{eqnarray*}
where
\begin{equation}\label{def-C1E}
	C_1(E) := \frac{
		\#\left\{
		M \in G_E(m_E): 
		\gcd(\tr M, m_E) = 1
		\right\}
	}{
		\#G_E(m_E)}.
\end{equation}
\item[(ii)]
Let $\ell$ be a prime such that $\ell \nmid  m_E$.
Assume  that there exists  some $\frac{1}{2} \leq \theta < 1$
such that
the $\theta$-quasi-GRH holds for the Dedekind zeta function of  $J_{E, \ell^2}$.
Then
\begin{eqnarray*}
\#\left\{a_p(E): p \leq x, p \nmid N_E, \gcd(a_p(E), m_E) = 1, a_p(E) \equiv 0 \left(\mod \ell^2\right)\right\}
&=&
\frac{1}{\ell^2 - 1}
C_1(E)
 \pi(x)
 \\
 &&
+
\O_E\left(\ell^4  x^{\theta} \log (\ell x)\right),
\end{eqnarray*}
with $C_1(E)$ defined  in (\ref{def-C1E}).
\end{enumerate}
\end{theorem}

\begin{proof}
Let $m$ be a positive integer with $\gcd(m, m_E) = 1$. Note that, in the notation $m = m_1 m_2$ of Proposition \ref{cheb-LaOd-Jm}, we have  $m_1 = 1$ and $m_2 = m$.
We want to estimate the cardinality of the set
\begin{equation}\label{definition-set-m}
{\cal{A}}_{E, m} :=
\left\{a_p(E): p \leq x, p \nmid N_E, \gcd(a_p(E), m_E) = 1, a_p(E) \equiv 0 \: (\mod m)\right\}
\end{equation}
when $m$ is one of the following two types:
either $m$ is an odd squarefree positive integer such that,
for some $\frac{1}{2} \leq \theta < 1$,
 the $\theta$-quasi-GRH holds for the Dedekind zeta function of $J_{E, m k}$ for all positive squarefree integers $k$ with $k \mid m_E$;
or  $m = \ell^2$ for some odd prime $\ell$ such that,
for some $\frac{1}{2} \leq \theta < 1$,
 the $\theta$-quasi-GRH holds  for the Dedekind zeta function of $J_{E, \ell^2}$.

Before making these particular choices of $m$, let us observe that

\begin{align}\label{size-of-Am-both}
\#{\cal{A}}_{E, m}
&=
\#\left\{a_p(E): p \leq x, p \nmid N_E, \gcd(a_p(E), m_E) = 1, a_p(E) \equiv 0 \: (\mod m)\right\}
\\
&=
\#\left\{a_p(E): p \leq x, p \nmid m N_E, \gcd(a_p(E), m_E) = 1, a_p(E) \equiv 0 \: (\mod m)\right\}
+ \O(\log m)
\nonumber
\\
&=
\Bigg(
\ds\sum_{
p \leq x
\atop{
p \nmid m N_E
\atop{
a_p(E) \equiv 0 \: (\mod m)
}
}
}
\ds\sum_{
k \geq 1
\atop{
k \mid \gcd(a_p(E), m_E)
}
}
\mu(k)
\Bigg)
+ \O(\log m)
\nonumber
\\
&=
\left(
\ds\sum_{
k \geq 1
\atop{
k \mid m_E
}
}
\mu(k)
\#\left\{p \leq x:
p \nmid m N_E, a_p(E) \equiv 0 \: (\mod m), a_p(E) \equiv 0 \: (\mod k) 
\right\}
\right)
+ \O(\log m)
\nonumber
\\
&=
\left(
\ds\sum_{
k \geq 1
\atop{
k \mid m_E
}
}
\left(
\mu(k)
\#\left\{p \leq x:
p \nmid m k  N_E, a_p(E) \equiv 0 \: (\mod m k)
\right\}
+ \O(\log k)
\right)
\right)
+ \O(\log m)
\nonumber
\\
&=
\left(
\ds\sum_{
k \geq 1
\atop{
k \mid m_E
}
}
\mu(k)
\#\left\{p \leq x:
p \nmid m k N_E, a_p(E) \equiv 0 \: (\mod m k)
\right\}
\right)
+
\O(\log x).
\nonumber
\end{align}
To pass to the second and fourth lines above, we used that for any positive integer $n$, $\omega(n) \leq 2\log n$. 
To pass to the fifth line above, we used that $\gcd(m, k) = 1$, since $k \mid m_E$ and $\gcd(m, m_E) = 1$.

 By  invoking  part (ii) of Proposition \ref{cheb-LaOd-Jm} under the assumption of a $\theta$-quasi-GRH for 
 $J_{E, m k}$ for all positive squarefree integers $k$ with $k \mid m_E$,
we see that (\ref{size-of-Am-both}) gives
\begin{equation}\label{Am-calculation}
\#{\cal{A}}_{E, m}
=
\frac{
\#\widehat{\cal{C}}(m, 0) 
}{
\#\PGL_2(\Z/m \Z)
}
\left(
\ds\sum_{
k \geq 1
\atop{
k \mid m_E
}
}
\mu(k)
\frac{
\#\widehat{\cal{C}}_E(k, 0)
\cdot
\#\Scal_{H_{E, k}}
}{
\#H_{E, k}
}
\right)
\
 \pi(x)
+
\O_E\left(
\#\widehat{\cal{C}}(m, 0) 
\ x^{\theta} \log (m N_E x)\right).
\end{equation}

Let us analyze the summation over $k \mid m_E$ occuring in (\ref{Am-calculation}).
We observe that, for each $k \mid m_E$, we have
\begin{equation*}\label{size-hat-C-k1-scal}
\#\widehat{\cal{C}}_E(k, 0)
\cdot
\#\Scal_{H_{E, k}}
=
\#\left\{
M \in G_E(k): \tr M = 0 \: (\mod k)
\right\}.
\end{equation*}
Furthermore, we observe that
\begin{align*}
\frac{
\#\left\{
M \in G_E(k): \tr M = 0 \: (\mod k)
\right\}
}{
\#H_{E, k}
}
&=
\frac{
\#\left\{
M \in G_E(k): \tr M = 0 \: (\mod k)
\right\}
}{
\# G_E(k)
}
\\
&=
\frac{
	\#\left\{
	M \in G_E(m_E): \tr M = 0 \: (\mod k)
	\right\}
}{
	\# G_E(m_E)
}.
\end{align*}
Then
\begin{eqnarray}\label{calculation-sum-over-k}
\ds\sum_{
k \geq 1
\atop{
k \mid m_E
}
}
\mu(k)
\frac{
\#\widehat{\cal{C}}_E(k, 0) 
\cdot
\#\Scal_{H_{E, k}}
}{
\#H_{E, k}
}
&=&
\frac{1}{\#G_E(m_E)}
\ds\sum_{
k \geq 1
\atop{
k \mid m_E
}
}
\mu(k) \#\left\{
M \in G_E(m_E): \tr M = 0 \: (\mod k)
\right\}
\\
&=&
\frac{
\#\left\{
M \in G_E(m_E): \tr M \not\equiv 0 \: (\mod \ell) \ \forall \ell \mid m_E
\right\}
}{
\#G_E(m_E)}
\nonumber
\\
&=&
C_1(E).
\nonumber
\end{eqnarray}

By plugging (\ref{calculation-sum-over-k}) in (\ref{Am-calculation}), we obtain that,
under the assumption of a $\theta$-quasi-GRH for $J_{E, m k}$ 
for all positive squarefree integers $k$ with $k \mid m_E$, we have
\begin{equation}\label{Am-calculation-simpler}
\#{\cal{A}}_{E, m}
=
\frac{
\#\widehat{\cal{C}}(m, 0) 
}{
\#\PGL_2(\Z/m \Z)
}
C_1(E)
 \pi(x)
+
\O_E\left(
\#\widehat{\cal{C}}(m, 0) 
\ x^{\theta} \log (m N_E x)\right).
\end{equation}

Finally, we specialize (\ref{Am-calculation-simpler}) to our two desired types of $m$.

\medskip

(i) In (\ref{Am-calculation-simpler}), we take $m = d$ for some odd squarefree positive integer $d$ coprime to $m_E$. 
The claimed estimate for $\#{\cal{A}}_{E, d}$ follows by invoking the Chinese Remainder Theorem
and by recalling that, from Lemma \ref{class-tr-0-sizes},
for any odd prime $\ell$, we have
$\#\widehat{\cal{C}}(\ell, 0)= \ell^2$.

\medskip

(ii) In (\ref{Am-calculation-simpler}), we take $m = \ell^2$ for some odd prime $\ell \nmid m_E$.
The claimed estimate for $\#{\cal{A}}_{E, \ell^2}$ follows 
by recalling that, from Lemma \ref{class-tr-0-sizes},
$\#\widehat{\cal{C}}\left(\ell^2, 0\right)= \ell^4$.
\end{proof}

\begin{theorem}\label{corollary-sizes-cond-sets-PCC}
Let $E$ be an elliptic curve defined over $\Q$, of conductor $N_E$,
 without complex multiplication,  and of torsion conductor $m_E$.
\begin{enumerate}
\item[(i)]
Let $d$ be a squarefree positive integer such that $\gcd(d,  m_E) = 1$.
Assume that GRH holds for the Dedekind zeta function of $J_{E, d k}$,
and that AHC and PCC hold for the extension $J_{E, dk}/\Q$,
 for all positive squarefree integers $k$ with $k \mid m_E$.
Then
\begin{eqnarray*}
\#\left\{a_p(E): p \leq x, p \nmid N_E, \gcd(a_p(E), m_E) = 1, a_p(E) \equiv 0 \: (\mod d)\right\}
&=&
\frac{1}{d}
\left(
\ds\prod_{\ell \mid d}
\left( 1 - \frac{1}{\ell^2}\right)^{-1}
\right)
C_1(E)
 \pi(x)
 \\&&
+
\O_E\left(x^{\frac{1}{2}} \log (d x)\right),
\end{eqnarray*}
where
$C_1(E)$
is defined in (\ref{def-C1E}).
\item[(ii)]
Let $\ell$ be a prime such that $\ell \nmid  m_E$.
Assume  that GRH holds for the Dedekind zeta function of  $J_{E, \ell^2}$,
and that AHC and PCC hold for the extension $J_{E, \ell^2}/\Q$
Then
\begin{eqnarray*}
\#\left\{a_p(E): p \leq x, p \nmid N_E, \gcd(a_p(E), m_E) = 1, a_p(E) \equiv 0 \left(\mod \ell^2\right)\right\}
&=&
\frac{1}{\ell^2 - 1}
C_1(E)
 \pi(x)
 \\
 &&
+
\O_E\left(x^{\frac{1}{2}} \log (\ell x)\right),
\end{eqnarray*}
with $C_1(E)$ defined  in (\ref{def-C1E}).
\end{enumerate}
\end{theorem}

\begin{proof}
We proceed similarly to the proof of Theorem \ref{corollary-sizes-cond-sets},
but invoke Proposition \ref{cheb-MuMuWo-Jm} instead of Proposition \ref{cheb-LaOd-Jm}.
Moreover,  
for any  squarefree positive integer $d$ such that $\gcd(d,  m_E) = 1$
and 
for any prime $\ell$ such that $\ell \nmid m_E$, 
we appeal to  the  upper bounds
$$
\frac{\#\PGL_2(\Z/d \Z)^{\#}}{\#\PGL_2(\Z/d \Z)} 
\ll
\frac{d}{d^3}
=
\frac{1}{d^2}
$$
and
$$
\frac{\#\PGL_2(\Z/\ell^2 \Z)^{\#}}{\#\PGL_2(\Z/\ell^2 \Z)} 
\ll
\frac{\ell^2}{\ell^6} 
= 
\frac{1}{\ell^2}.
$$
\end{proof}

We end this section with an immediate application of  part (i) of Theorem \ref{cheb-LaOd},
which we will need in the proof of our two main theorems.

\begin{lemma}\label{Lang-Trotter-bound-quasi-GRH}
Let $E$ be an elliptic curve defined over $\Q$, without complex multiplication, of conductor $N_E$, and of torsion conductor $m_E$.
Assume  that there exists  some $\frac{1}{2} \leq \theta < 1$
such that
the $\theta$-quasi-GRH holds for the Dedekind zeta functions of the division fields of $E$.
Then, for any $\alpha \in \Z$ with $\alpha \neq 0$, 
\begin{equation}\label{upper-bound-quasiGRH-nonzero-trace}
\#\left\{p \leq x: p \nmid N_E, a_p(E) = \alpha\right\} \ll_E \frac{x^{1 - \frac{1-\theta}{4}}}{(\log x)^{\frac{1}{2}} }
\end{equation}
and
\begin{equation}\label{upper-bound-quasiGRH-zero-trace}
\#\left\{p \leq x: p \nmid N_E, a_p(E) = 0\right\} \ll_E \frac{x^{1 - \frac{1-\theta}{3}}}{(\log x)^{\frac{1}{3}} }.
\end{equation}
\end{lemma} 
\begin{proof}
Let $\ell$ be a prime such that $\ell \nmid m_E$. Then
$$
\#\left\{p \leq x: p \nmid N_E, a_p(E) = \alpha\right\}
\leq
\#\left\{p \leq x: p \nmid N_E, a_p(E) \equiv \alpha \: (\mod \ell)\right\}.
$$
For the left hand side of the above inequality, we invoke part (i) of Proposition \ref{cheb-LaOd-Jm},
together with (\ref{matrix-count-mod-ell-without-hat}) of Lemma \ref{class-tr-0-sizes},
and deduce that
\begin{equation*}
\#\left\{p \leq x: p \nmid N_E, a_p(E) \equiv \alpha \: (\mod \ell)\right\}
\ll_E
\frac{x}{\ell \log x} + \ell^3 x^{\theta} \log (\ell x).
\end{equation*}
We choose
$\ell \asymp \frac{x^{ \frac{1-\theta}{4} }}{(\log x)^{\frac{1}{2}} }$
and obtain  the desired upper bound for
$\#\left\{p \leq x: p \nmid N_E, a_p(E) = \alpha\right\}$.

When $\alpha = 0$, the result can be strengthened by 
invoking part (ii) of Proposition \ref{cheb-LaOd-Jm}
and 
(\ref{matrix-count-mod-ell-with-hat}) of Lemma \ref{class-tr-0-sizes},
leading to the upper bounds
\begin{equation*}
\#\left\{p \leq x: p \nmid N_E, a_p(E) = 0 \right\}
\leq
\#\left\{p \leq x: p \nmid N_E, a_p(E) \equiv 0 \: (\mod \ell)\right\}
\ll_E
\frac{x}{\ell \log x} + \ell^2 x^{\theta} \log (\ell x).
\end{equation*}
In this case, we choose
$\ell \asymp \frac{x^{ \frac{1-\theta}{3} }}{(\log x)^{\frac{2}{3}} }$
and obtain the desired upper bound for
$\#\left\{p \leq x: p \nmid N_E, a_p(E) = 0\right\}$.
\end{proof}

\begin{remark}
{\emph{
Much better conditional upper bounds for 
$\#\left\{p \leq x: p \nmid N_E, a_p(E) = \alpha\right\}$ 
are known for a non-zero integer $\alpha$
and
 the much stronger unconditional result  
 $\#\left\{p \leq x: p \nmid N_E, a_p(E) = 0\right\} \ll_E x^{\frac{3}{4}}$
 is also known.
For the purpose of Theorems \ref{main-thm-upper} and \ref{main-thm-lower}, the weaker upper bound 
(\ref{upper-bound-quasiGRH-nonzero-trace})
of Lemma \ref{Lang-Trotter-bound-quasi-GRH}, under the assumption of a $\theta$-quasi-GRH 
and not of the full GRH,
suffices. We recorded this weaker upper bound to be able to refer to it in the upcoming sections.
}}
\end{remark}

\section{Sieve setting}\label{sec_sieve_setting}

\subsection{General sieve setting}\label{sec_sieve_setting-general}

Inspired by classical approaches towards the twin prime conjecture,
we study the primality of the Frobenius traces associated to an elliptic curve 
using sieve methods. The general  sieve setting that we will be using is as follows:

\medskip

$\bullet$
${\cal{A}}$  is a finite sequence of integers;

$\bullet$
${\cal{P}}$ is  a set of primes;


$\bullet$
for each prime $\ell \in {\cal{P}}$,  
${\cal{A}}_{\ell}$ and ${\cal{A}}_{\ell^2}$ are  defined by
 $${\cal{A}}_{\ell} := \{a \in {\cal{A}}: a \equiv 0 (\mod \ell)\}$$
and  
$${\cal{A}}_{\ell^2} := \{a \in {\cal{A}}: a \equiv 0 (\mod \ell^2)\};$$

$\bullet$
for each positive squarefree  integer $d$ composed of primes of ${\cal{P}}$,
${\cal{A}}_{d}$ is  defined by
$${\cal{A}}_d := \ds\bigcap_{\ell \mid d} {\cal{A}}_{\ell} = \{a \in {\cal{A}}: a \equiv 0 (\mod d)\};$$

$\bullet$
$z, z_1, z_2 > 0$ are positive real numbers;

$\bullet$
$P(z)$ is the positive integer defined by
$$P(z) := \ds\prod_{\ell \in {\cal{P}} \atop{\ell < z}} \ell;$$

$\bullet$
$S({\cal{A}}, {\cal{P}}, z)$ is the cardinality of the sifted set 
${\cal{A}} \backslash \left(\ds\cup_{d \mid P(z)} {\cal{A}}_d\right)$, that is, 
$$S({\cal{A}}, {\cal{P}}, z) := \#\left\{a \in {\cal{A}}: \gcd(a, P(z)) = 1\right\}.$$

\medskip

The general sieve problem is that of estimating the cardinality
$S({\cal{A}}, {\cal{P}}, z)$
under additional assumptions.
In particular, 
for the two sieves that we will be using,
the Selberg Upper Bound Sieve and the Greaves Lower Bound Sieve
(see Section \ref{sec_proof_upper} and Section \ref{sec_proof_lower}, 
respectively), 
we will make the assumption that
there exist 
a real number $X > 0$
and
 a non-zero multiplicative function
$w: \N \backslash\{0\} \longrightarrow \R $
such that,
for any positive squarefree integer $d$ composed of primes of ${\cal{P}}$, 
we have
\begin{equation}\label{hypothesis-R-d}
 \#{\cal{A}}_d = \frac{w(d)}{d} X + R_d
 \
 \text{ for some}
 \ 
 R_d \in \R.
\end{equation}
For the Greaves Lower Bound Sieve, 
in addition to (\ref{hypothesis-R-d}),
we will make the assumption that,
for any prime $\ell \in {\cal{P}}$,
we have
\begin{equation}\label{hypothesis-R-ell-squared}
 \#{\cal{A}}_{\ell^2} = \frac{w(\ell^2)}{\ell^2} X + R_{\ell^2}
 \
 \text{ for some}
 \ 
 R_{\ell^2} \in \R.
\end{equation}
Moreover, for both sieves, we will make the  assumptions that
	there exists $\e > 0$ such that,
	for any prime $\ell \in \cP$,  we have
		\begin{equation}\label{w-function-asmptn-1}
			0 \leq \frac{w(\ell)}{\ell} \leq 1 - \e,
		\end{equation}
	and that there exist $L, A \geq 1$ such that, for any $z_1, z_2$ with $2 \leq z_1 < z_2$, we have
		\begin{equation}\label{w-function-asmptn-2}
			-L 
			\leq 
			\ds\sum_{z_1\leq \ell < z_2 \atop{\ell \in {\cal{P}} }}
			\frac{w(\ell)}{\ell}\log \ell - \log \frac{z_2}{z_1} \leq A.
		\end{equation}

We note that, for each of the two sieves, the estimate for the cardinality of the sifted set will be formulated in terms of
the function
\begin{equation}\label{def-W}
W(z) := \ds\prod_{\ell \mid P(z)} \left(1 - \frac{w(\ell)}{\ell}\right).
\end{equation}

\subsection{Elliptic curve sieve setting}\label{sec_sieve_setting-ec}

To prove Theorems \ref{main-thm-upper} and \ref{main-thm-lower}, we will use
the Selberg Upper Bound  Sieve and the Greaves Lower Bound Sieve, respectively,
in the following setting. 

We fix an elliptic curve $E$ defined over $\Q$, without complex multiplication, of conductor $N_E$, and of torsion conductor $m_E$, and we assume that there exists  some $\frac{1}{2} \leq \theta < 1$ such that the $\theta$-quasi-GRH holds for
the Dedekind zeta functions of 
 $\Q(E[m])$ and $J_{E, m}$ for all positive integers $m$. We fix $x > 2$ (to be thought of as going to infinity)
and  set
\begin{align*}
{\cal{A}} = {\cal{A}}_E &:= \left\{a_p(E): p \leq x, p \nmid N_E, \gcd(a_p(E), m_E) = 1\right\},
\\
{\cal{P}} = {\cal{P}}_E &:=  \left\{\ell: \ \ell \nmid m_E\right\}.
\end{align*}
Recalling that $m_E$ is even, we see that all the primes $\ell \in {\cal{P}}$ are odd. Consequently, all the squarefree integers composed of primes of ${\cal{P}}$ are also odd.

With these definitions, we see that, for each positive squarefree $d$ with $\gcd(d,m_E) = 1$
and for each prime $\ell \nmid m_E$,
we have
\begin{equation*}\label{A_d-in-setting}
	\cA_d = \cA_{E, d} = \{a_p(E) : p \leq x, p \nmid N_E, \gcd(a_p(E), m_E) = 1, a_p(E) \equiv 0 \: (\mod d)\},
\end{equation*}
\begin{equation*}\label{A_l^2-in-setting}
	\cA_{\ell^2} = \cA_{E, \ell^2} = \{a_p(E) : p \leq x, p \nmid N_E, \gcd(a_p(E), m_E) = 1, a_p(E) \equiv 0 \: (\mod \ell^2)\}.
\end{equation*}

In order to check the sieve assumptions mentioned in Subsection \ref{sec_sieve_setting-general}, 
it remains to identify $X$ and $w(\cdot)$,
and to bound $|R_{d}|$ and $|R_{\ell^2}|$ from above, which is what we do next.

Noting that ${\cal{A}}_d$ and ${\cal{A}}_{\ell^2}$ are the sets introduced in 
the proof of 
Theorem \ref{corollary-sizes-cond-sets},
we deduce that
\begin{equation}\label{sieve-condition-A-d}
\#{\cal{A}}_{d}
=
\frac{1}{d}
\left(
\ds\prod_{\ell \mid d}
\left( 1 - \frac{1}{\ell^2}\right)^{-1}
\right)
C_1(E)
 \pi(x)
+
\O_E\left(d^2 x^{\theta} \log (d x)\right),
\end{equation}
\begin{equation*}
\#{\cal{A}}_{\ell^2}
=
\frac{1}{\ell^2 - 1}
C_1(E)
 \pi(x)
+
\O_E\left(\ell^4  x^{\theta} \log (\ell x)\right),
\end{equation*}
with $C_1(E)$  defined in (\ref{def-C1E}).
We conclude that, in our particular sieve setting, we may  take
\begin{equation}
\label{def-X}
X 
:=
C_1(E)
 \pi(x)
\end{equation}
and
 \begin{equation}
 \label{definition-density-d}
 w(d) 
 :=
 \ds\prod_{\ell \mid d}
 \left(1 - \frac{1}{\ell^2}\right)^{-1},
\end{equation}
in which case
\begin{equation}\label{upper-bound-R-d}
|R_d| \ll_E d^2 x^{\theta} \log (d x)
\end{equation}
and
\begin{equation}\label{upper-bound-R-ell-squared}
\left|R_{\ell^2}\right| \ll_E \ell^4 x^{\theta} \log (\ell x).
\end{equation}
Let us point out that, in (\ref{upper-bound-R-d}) and (\ref{upper-bound-R-ell-squared}),
the exponent  $\theta$  reflects the assumption of the $\theta$-quasi-GRH.

If we assume that $\theta = \frac{1}{2}$
 and if, in addition, we assume that AHC and PCC hold for the Artin L-functions of the extensions 
  $J_{E, m}/\Q$ for all positive integers $m$,
  then, 
 using Theorem  \ref{corollary-sizes-cond-sets-PCC}, 
 we obtain that
\begin{equation}\label{upper-bound-R-d-PCC}
|R_d| \ll_E  x^{\frac{1}{2}} \log (d x)
\end{equation}
and
\begin{equation}\label{upper-bound-R-ell-squared-PCC}
\left|R_{\ell^2}\right| \ll_E  x^{\frac{1}{2}} \log (\ell x).
\end{equation}

\subsection{Sieve assumptions for the elliptic curve sieve setting}\label{sec_sieve_setting-ec-assumptions}

We will verify that the sieve assumptions (\ref{w-function-asmptn-1}) and (\ref{w-function-asmptn-2})
mentioned previously
are satisfied. 

It is easy to see that the function $w(\cdot)$ is  decreasing on prime values. Hence, for any prime $\ell$, we have
$0 \leq \frac{w(\ell)}{\ell} \leq \frac{2}{3}$.
As such, the first assumption  (\ref{w-function-asmptn-1}) is satisfied. 

For the second assumption, fix $z_1, z_2$ with $2 \leq z_1 < z_2$. Then
\begin{align}
	\sum_{z_1\leq \ell < z_2}\frac{w(\ell)}{\ell}\log \ell
	&= \sum_{z_1\leq \ell < z_2}\frac{\left(1 - \frac{1}{\ell^2}\right)^{-1}}{\ell}\log \ell
	\nonumber
	\\	
	&= \sum_{z_1\leq \ell < z_2}\frac{\ell}{\ell^2-1}\log \ell
	\nonumber
	\\ 
	&= \sum_{z_1\leq \ell < z_2}\frac{\log \ell}{\ell}
			+\sum_{z_1\leq \ell < z_2}\frac{\log \ell}{\ell(\ell^2-1)}. \label{w-function-asmpn-2-calculation}
\end{align}
Using Mertens' First Theorem,
\begin{equation}\label{Mertens'-1st}
\left|
\ds\sum_{\ell \leq y} \frac{\log \ell}{\ell} - \log y
\right|
\leq 2
\ \forall y > e,
\end{equation}
we  see that the first sum in line (\ref{w-function-asmpn-2-calculation}) differs from $\log\frac{z_2}{z_1}$ by at most $4$. Extending the range of the second sum to all primes $\ell \geq 2$ yields a series that converges to a value less than $1$.
 As such,
\begin{equation}\label{w-function-asmpn-2-verification}
	\left| \sum_{z_1\leq \ell < z_2}\frac{w(\ell)}{\ell}\log \ell - \log \frac{z_2}{z_1} \right| < 5,
\end{equation}
which confirms that the second assumption (\ref{w-function-asmptn-2}) holds.

\subsection{Main term estimate for the elliptic curve sieve setting}\label{sec_sieve_setting-ec-main-t}
We will now estimate the function $W(z)$  introduced in (\ref{def-W}). 
For this, we need Mertens' Third Theorem,
\begin{equation}\label{Mertens'-3rd}
\ds\lim_{y \rightarrow \infty} (\log y) \ds\prod_{\ell \leq y} \left(1 - \frac{1}{\ell}\right) = e^{-\gamma},
\end{equation}
where $\gamma$ is Euler's constant.
We also need the following property of convergent products.
\begin{lemma}\label{convergent-product-property}
Suppose a series $\ds\sum_{n\geq1} a_n$ converges absolutely. 
Define $F(y) := \ds\sum_{n \geq y} \left| a_n\right|$. Then, for large enough $y$,
\begin{equation*}
\ds\prod_{n \geq y} \left(1+a_n \right) = 1 + \O(F(y)).
\end{equation*}
\end{lemma}
\begin{proof}
We start by taking the logarithm of the product. Then, provided $y$ is large enough, we are guaranteed to have 
$|a_n| < 1$ and we may rewrite $\log(1+a_n)$ as a power series.
We obtain
\begin{align*}
		\left|\log \ds\prod_{n \geq y} \left(1+a_n \right) \right|
			= \left| \ds\sum_{n \geq y} \log (1+a_n) \right|
			= \left| \ds\sum_{n \geq y} \sum_{k \geq1} -\frac{(-a_n)^k}{k} \right|
			\leq \ds\sum_{n \geq y} \sum_{k \geq 1} \left|a_n \right|^k
			\ll F(y).
\end{align*}
Note that $F(y) = \o(1)$ since we assumed  that $\ds\sum_{n \geq 1} a_n$ converges absolutely.
Then, putting everything together, 
we obtain that 
$\ds\prod_{n \geq y} (1+a_n) = e^{\O(F(y))} = 1 + \O(F(y))$.
\end{proof}

We are now ready to analyze the function $W(z)$ defined in (\ref{def-W}).
Using (\ref{definition-density-d}), for any $z > m_E$ we obtain that
\begin{eqnarray}\label{W(z)-ec-setting}
W(z)
&=&
\ds\prod_{
\ell < z
\atop{\ell \nmid m_E}
}
 \left(1 -  \ell^{-1}
 \left(
 1 - \frac{1}{\ell^2}
 \right)^{-1}\right)
 \nonumber
 \\
&=&
\left(
\ds\prod_{\ell < z \atop{\ell \nmid m_E}}
\left(
1 - \frac{1}{\ell}
\right)
\right)
\cdot
\left(
\ds\prod_{\ell< z \atop{\ell \nmid  m_E}}
\left(
1 - \frac{1}{\ell^3 - \ell^2 - \ell +1}
\right)
\right)
\nonumber
\\
&=&
\left(
\ds\prod_{\ell < z \atop{\ell | m_E}}
\left(
1 - \frac{1}{\ell}
\right)^{-1}
\right)
\cdot
\left(
\ds\prod_{\ell < z}
\left(
1 - \frac{1}{\ell}
\right)
\right)
\cdot
\left(
\ds\prod_{\ell \nmid  m_E}
\left(
1 - \frac{1}{\ell^3 - \ell^2 - \ell +1}
\right)
\right)
\cdot
\left(
\ds\prod_{\ell \geq z}
\left(
1 - \frac{1}{\ell^3 - \ell^2 - \ell +1}
\right)
\right)^{-1}
\nonumber
\\
&=&
\frac{m_E}{\phi(m_E)}
\cdot
\left(
\ds\prod_{\ell < z}
\left(
1 - \frac{1}{\ell}
\right)
\right)
\cdot
\left(
\ds\prod_{\ell \nmid  m_E}
\left(
1 - \frac{1}{\ell^3 - \ell^2 - \ell +1}
\right)
\right)
\cdot
\left(
\ds\prod_{\ell \geq z}
\left(
1 + \frac{1}{\ell^3 - \ell^2 - \ell}
\right)
\right)
\nonumber
\\
&=&
\frac{m_E}{\phi(m_E)}
\cdot
\left(
\ds\prod_{\ell \nmid  m_E}
\left(
1 - \frac{1}{\ell^3 - \ell^2 - \ell +1}
\right)
\right)
\cdot
\left(
\ds\prod_{\ell < z}
\left(
1 - \frac{1}{\ell}
\right)
\right)
\cdot
\left(1 + \O\left(\frac{1}{z^2}\right)\right)
\nonumber
\\
&=&
\frac{m_E}{\phi(m_E)}
\cdot
\left(
\ds\prod_{\ell \nmid  m_E}
\left(
1 - \frac{1}{\ell^3 - \ell^2 - \ell +1}
\right)
\right)
\cdot
\left(
\frac{e^{-\gamma}}{\log z} + \o\left(\frac{1}{\log z}\right)
\right).
\nonumber
\end{eqnarray}
To pass from the fourth line to the fifth, we used Lemma \ref{convergent-product-property}.
To pass from the fifth line to the sixth, we used Mertens' Third Theorem (\ref{Mertens'-3rd}).

For later purposes, we record the above calculation as 
\begin{equation}\label{W(z)-growth}
W(z) 
=
C_2(E) \cdot
\left(
\frac{e^{-\gamma}}{\log z} + o\left(\frac{1}{\log z}\right)
\right),
\end{equation}
where
\begin{equation}\label{constant-of-E}
C_2(E) := 
\frac{m_E}{\phi(m_E)}
\cdot
\ds\prod_{\ell \nmid  m_E}
\left(
1 - \frac{1}{\ell^3 - \ell^2 - \ell +1}
\right).
\end{equation}

\section{Proof of Theorem \ref{main-thm-upper} and Corollary \ref{Brun-analog} }\label{sec_proof_upper}

To prove Theorem \ref{main-thm-upper}, 
we will use a simplified version of the Selberg Upper Bound Sieve, as presented in \cite[Thm. 8.3, p. 231]{HaRi74}.

\begin{theorem}\label{Selberg-theorem} 
Let ${\cal{A}}$ be a finite sequence of integers and let ${\cal{P}}$ be a set of primes.
Use the notation  ${\cal{A}}_{\ell}$, ${\cal{A}}_d$, $P(z)$, $S(\cal{A}, \cal{P}, z)$, and $W(z)$ 
introduced in Subsection  \ref{sec_sieve_setting-general}.
Then, under assumptions (\ref{hypothesis-R-d}),  (\ref{w-function-asmptn-1}) and (\ref{w-function-asmptn-2}),
 there exists $B > 0$ such that, for any $z > 0$,
\begin{equation}\label{Selberg-conclusion}
S(\cA,\cP,z) 
\leq 
X W(z)
\left(
e^\gamma + \frac{BL}{(\log z)^{1/14}} 
\right)
+
\ds\sum_{
\substack{d\leq z^2 \\ d \mid P(z)}} 
3^{\omega(d)}\lvert R_d \rvert.
\end{equation}
\end{theorem}

When applying the above theorem to our particular elliptic curve setting, 
 the following estimate will be helpful.

\begin{lemma}\label{b^omg-average-lemma}
	Let $a \in (-1, \infty)$  and  $k \in \N \setminus \{0\}$. For each $y > e$, we have
	\begin{equation}\label{lemma-goal}
	\ds\sum_{n\leq y} n^a k^{\omega(n)} \ll_{a, k} y^{a+1}(\log y)^{k-1}.
	\end{equation}
\end{lemma}
\begin{proof}
	We will first prove the estimate for $a=0$ by proceeding by induction on $k$. The base case $k=1$ is clear. Note that, for any $k \in \N \setminus \{0\}$,
	\begin{equation}\label{s^Omega-formula}
		k^{\omega(n)} \leq \ds\sum_{d_1d_2...d_k = n} 1.
	\end{equation}
	Now, assume that (\ref{lemma-goal}) holds for $a=0$ and some fixed $k$. Then we see from (\ref{s^Omega-formula})  that
	\begin{align*}
		\ds\sum_{n\leq y} (k+1)^{\omega(n)}
			&\leq \ds\sum_{n\leq y} \;\; \ds\sum_{d_1...d_{k+1} = n} 1
	\\		&= \ds\sum_{d_{k+1} \leq y} \;\;
				\ds\sum_{\substack{n \leq y \\ d_{k+1} \mid n}} \;\;
				\ds\sum_{d_1...d_k = \frac{n}{d_{k+1}}} 1
	\\  	&= \ds\sum_{d_{k+1} \leq y} \;\;
				\ds\sum_{m \leq \frac{y}{d_{k+1}}} \;\;
				\ds\sum_{d_1...d_k = m} 1
	\\		&\ll_{k} \ds\sum_{d_{k+1} \leq y} \frac{y}{d_{k+1}}
				\left(\log \frac{y}{d_{k+1}} \right)^{k-1}
	\\ 	 	&\ll_{k} y (\log y)^k.
	\end{align*}
	This completes the induction for $a=0$. 
	To prove the estimate for $a \neq 0$, we start by fixing  $a > -1$ and  $k \in \N \setminus \{0\}$. 
	Then, using partial summation, we obtain
	\begin{align*}
		\ds\sum_{n\leq y} n^a k^{\omega(n)}
			&= y^a \ds\sum_{n\leq y} k^{\omega(n)} - a\int_{1}^{y} t^{a-1} \ds\sum_{n\leq t} k^{\omega(n)} dt
	\\		&\ll_{a} y^{a+1}(\log y)^{k-1} + \int_{1}^{y} t^a (\log t)^{k-1} dt.
	\end{align*}
	The integral above can be evaluated through repeated uses of integration by parts,
	leading to the upper bound
	  $\ll_{a, k} y^{a+1} (\log y)^{k-1}$.
\end{proof}

\medskip

\begin{remark}\label{ap-multiset}
{\emph{
Before we begin proving Theorem \ref{main-thm-upper}, it is worth remarking on an interesting wrinkle that arises from the set of Frobenius traces $a_p(E)$ being 
a 
multiset, 
i.e. displaying the feature that certain values of $a_p(E)$ repeat for different values of $p$. Theorem \ref{Selberg-theorem} 
detects each $a \in \cA$ whose only prime factors are large
and thus
provides an upper bound for the number of large primes appearing in $\cA$.
In particular, it gives  no information about the small primes appearing in $\cA$. 
When $\cA$ is a set (that is, has no repeated elements), this is not a problem, since we may write,
for any $z > 0$,
\begin{align*}
	\# \{a\in \cA : a \text{ prime} \}
	= \# \{a\in \cA : a \text{ prime}, \lvert a \rvert < z \}
	+ \# \{a\in \cA : a \text{ prime}, \lvert a \rvert \geq z \} 
	\leq 2z + \mathcal{S}(\cA,\cP,z).
\end{align*}
In this case, choosing $z$ of negligible size, the sieve on its own leads to an upper bound for the number of primes appearing in $\cA$. In contrast, when $\cA$ is a multiset, since we cannot bound $\# \{a\in \cA : a \text{ prime}, |a| < z \}$ by $2z$
 (for $\cA$ may contain 
some small prime
infinitely many times),
 the sieve itself is not enough to bound the number of primes in $\cA$
without
additional information 
about $\cA$.
 For our particular elliptic curve setting,
 this additional information is derived from Lemma \ref{Lang-Trotter-bound-quasi-GRH}. 
 }}
 \end{remark}

\medskip

\begin{proof}[ {\bf{Proof of Theorem \ref{main-thm-upper} }}]
We fix an elliptic curve $E$ defined over $\Q$, without complex multiplication, of conductor $N_E$, and of torsion conductor $m_E$.
We assume that
there exists  some $\frac{1}{2} \leq \theta < 1$
such that
the $\theta$-quasi-GRH holds for the Dedekind zeta functions of 
 $\Q(E[m])$ and $J_{E, m}$ for all positive integers $m$.
We fix $x > 0$, to be thought of as going to infinity.
As in Subsection \ref{sec_sieve_setting-ec},
we define
\begin{equation*}
\cA := \{a_p(E): p \leq x, p \nmid N_E, \gcd(a_p(E),m_E) = 1 \},
\end{equation*}
\begin{equation*}
\cP := \{\ell \: \text{prime}: \ell \nmid m_E \}.
\end{equation*}
We recall that $m_E$ is even, which implies that all primes $\ell \in \cP$ are odd.

With these choices, we showed in 
(\ref{sieve-condition-A-d}), (\ref{def-X}), (\ref{definition-density-d}), (\ref{upper-bound-R-d}),  and (\ref{W(z)-growth})
that 
\begin{equation*}
\#{\cal{A}}_{d}
=
\frac{1}{d}
\left(
\ds\prod_{\ell \mid d}
\left( 1 - \frac{1}{\ell^2}\right)^{-1}
\right)
C_1(E)
 \pi(x)
+
\O_E\left(d^2 x^{\theta} \log (d x)\right),
\end{equation*}
\begin{equation*}
X 
=
C_1(E)
 \pi(x),
\end{equation*}
 \begin{equation*}
  w(d) 
 =
 \ds\prod_{\ell \mid d}
 \left(1 - \frac{1}{\ell^2}\right)^{-1},
\end{equation*}
\begin{equation*}
|R_d| \ll_E d^2 x^{\theta} \log (d x),
\end{equation*}
and that, for any $z > m_E$,
\begin{equation*}
W(z) 
=
C_2(E) \cdot
\left(
\frac{e^{-\gamma}}{\log z} + \o\left(\frac{1}{\log z}\right)
\right).
\end{equation*}	
Furthermore, we showed that $w(\cdot)$ satisfies  assumptions (\ref{w-function-asmptn-1}) and (\ref{w-function-asmptn-2}),
which are needed in order to apply
Theorem \ref{Selberg-theorem}.

Now let $z = z(x) > m_E$ be a parameter of $x$, to be chosen optimally later. 
Recalling the definition of $\mathcal{S}(\cA,\cP,z)$
and invoking Lemma \ref{Lang-Trotter-bound-quasi-GRH} and Theorem \ref{Selberg-theorem},
we obtain that
\begin{eqnarray}
\pi_{E, \text{prime trace}}(x)
&=&
 \#\{a_p(E): p \leq x, p \nmid N_E, a_p(E) \; \text{prime}\} 
 \nonumber
 \\
&=&
 \# \{a_p(E): p \leq x, p\nmid N_E, a_p(E) \text{ prime}, \lvert a_p(E)\rvert \geq z \}
 \nonumber
 \\
&+&
\# \{a_p(E): p \leq x, p\nmid N_E, a_p(E) \text{ prime}, \lvert a_p(E)\rvert < z \} 
\nonumber
\\
&\leq&
 S(\cA,\cP,z)
+ 
\O_E \left( \frac{x^{1 - \frac{1-\theta}{4}}z}{(\log x)^{\frac{1}{2}}}\right)  
\nonumber
\\
&\leq& 
X W(z)
\left(e^\gamma + \frac{5B}{(\log z)^{1/14}} \right)
+ \ds\sum_{\substack{d\leq z^2 \\ \gcd(d,m_E) = 1}} 
3^{\omega(d)}\lvert R_d \rvert
+ 
\O_E \left( \frac{x^{1 - \frac{1-\theta}{4}}z}{(\log x)^{\frac{1}{2}}}\right) 
\label{upper-initial-inequality}
\end{eqnarray}
for some absolute constant $B > 0$.

	In order for the last inequality to be meaningful, we want to ensure that its last two terms  are
	$\o\left(\frac{x}{(\log x)^2}\right)$. We claim that this is the case if we choose
	\begin{equation}\label{choice-z-Selberg}
		z := \frac{x^{\frac{1-\theta}{6}}}{(\log x)^2}.
	\end{equation}
	
	Using  bound (\ref{upper-bound-R-d}) for $R_d$ and Lemma \ref{b^omg-average-lemma}, we obtain that
	\begin{align*}
		\ds\sum_{\substack{d\leq z^2 \\ \gcd (d,m_E) = 1}} 3^{\omega(d)}\lvert R_d \rvert
			\ll _E \ds\sum_{d \leq z^2}d^2 3^{\omega(d)}x^\theta\log x
			\ll x^\theta z^6\log x (\log z)^2
			\ll \frac{x}{(\log x)^9}.
	\end{align*}
	Then, using (\ref{choice-z-Selberg}), we see  that
	\begin{equation*}
		\frac{x^{1 - \frac{1-\theta}{4}}z}{(\log x)^{\frac{1}{2}}} \ll \frac{x^{1-\frac{1-\theta}{12}}}{(\log x)^{5/2}}.
	\end{equation*}
	As such, choice (\ref{choice-z-Selberg}) of $z$ makes the last two terms in  (\ref{upper-initial-inequality})
	be  $\o\left(\frac{x}{(\log x)^2}\right)$. 
	
	It remains to examine the first  term in (\ref{upper-initial-inequality}).
	Recalling  (\ref{def-X}) and (\ref{W(z)-growth}),
	we see that 
	\begin{align}
		X W(z) e^\gamma
			= 
			C_1(E)C_2(E)\pi(x)
			\left(
			\frac{1}{\log z} + \o\left(\frac{1}{\log z}\right)
			\right)
			= \left(\frac{3}{1-\theta} + \o(1)\right) C(E) \frac{x}{(\log x)^2}, \label{upper-main-term}
	\end{align}
	where
	$C(E)$ is as in (\ref{conjecture-constant}).
	
	Altogether, by using (\ref{upper-main-term}) in (\ref{upper-initial-inequality}) and by gathering all the error terms into the little $\o$-notation, we obtain that
	\begin{equation*}
		\pi_{E, \text{prime trace}}(x)
		\leq \left(\frac{3}{1-\theta} + \o(1)\right) C(E) \frac{x}{(\log x)^2}.
	\end{equation*}
This completes the proof of Theorem \ref{main-thm-upper}.
\end{proof}

\bigskip

From Theorem \ref{main-thm-upper}, we can derive 
the convergence of the sum of the reciprocals of the primes $p$ with the property that the Frobenius trace $a_p(E)$ is also a prime.

\begin{proof}[ {\bf{Proof of Corollary \ref{Brun-analog} }}]
	Fix $\e_0 > 0$. By Theorem \ref{main-thm-upper}, 
	there exists $x_0 = x_0(E,\theta,\e_0)$ such that, for all $x \geq x_0$,
	\begin{equation*}
	\pi_{E, \text{prime trace}}(x)
				\leq \left(\frac{3}{1-\theta} + \e_0\right) C(E) \frac{x}{(\log x)^2}.
	\end{equation*}
	Then, by using partial summation and the above inequality, we deduce that
	\begin{align}
		\ds\sum_{\substack{{p \geq x_0}\\{a_p \text{ prime}}}} \frac{1}{p}
			&= \frac{\pi_{E, \text{prime trace}}(t)}{t}\bigg|_{x_0}^\infty 
			      + 
			     \int_{x_0}^{\infty} \frac{\pi_{E, \text{prime trace}}(t)}{t^2} dt \nonumber
			 \\
			&\leq \left(\frac{3}{1-\theta} + \e_0\right) C(E) \frac{1}{\log x_0} - \frac{\pi_{E, \text{prime trace}}(x_0)}{x_0}  \nonumber \\
			&\leq \left(\frac{3}{1-\theta} + \e_0\right) C(E) \frac{1}{\log x_0}. \nonumber
	\end{align}
	\end{proof}

\section{Proof of Theorem \ref{main-thm-lower}}\label{sec_proof_lower}

Our main tool in the proof of Theorem \ref{main-thm-lower} is the weighted Greaves Lower Bound Sieve of
dimension one, proven in \cite{HaRi85} and recalled below. 

In the general sieve setting of Subsection \ref{sec_sieve_setting-general},
we assume
(\ref{hypothesis-R-d}), 
(\ref{hypothesis-R-ell-squared}),
(\ref{w-function-asmptn-1}),
and
(\ref{w-function-asmptn-2}),
along with
\begin{eqnarray}\label{A0-HRIp157-lower}
0 \leq w(\ell)  < \ell \quad \forall \ell \in {\cal{P}}.
\end{eqnarray}

We set
$$
V_0 := 0.074368...
$$
as in \cite[(6.16), p. 205]{HaRi85}
and  fix positive real numbers $U$ and $V$ such that
\begin{equation}\label{U-V-one}
V_0 < V < U,
\end{equation}
\begin{equation}\label{U-V-two}
V \leq \frac{1}{4},
\end{equation}
\begin{equation}\label{U-V-three}
U \geq \frac{1}{2},
\end{equation}
\begin{equation}\label{U-V-four}
U + 3 V \geq 1. 
\end{equation}
We  fix
$$
z > 0
$$
and
consider an arithmetic function
$$
g(\cdot) : \N \backslash \{0\} \longrightarrow \R
$$
having the property that, on primes $\ell \in {\cal{P}}$,
\begin{equation}\label{arithmetic-fcn-w-ell}
g(\ell)
:=
\left\{
\begin{array}{cc}
\frac{1}{U-V} \left(\frac{\log \ell}{\log z} - V\right)  & \text{if} \ z^{V} \leq \ell < z^U,
\\
0 &  \text{otherwise}.
\end{array}
\right.
\end{equation}
Note that
\begin{equation}\label{arithmetic-fcn-w-1}
0 \leq g(\ell) \leq 1 \quad \forall \ell \in {\cal{P}}
\end{equation}
and
\begin{equation}\label{arithmetic-fcn-w-1}
 g(\ell) = 0 \quad \forall \ell \in {\cal{P}} \ \text{with}\ \ell < z^V.
\end{equation}
Using the notation
\begin{equation}\label{def-u-plus}
\{u\}^{+} := \max\{u, 0\}
\end{equation}
for any given $u \in \R$, we define 
$${\cal{G}}(\cdot): \N \backslash \{0\} \longrightarrow \R,$$
$$
{\cal{G}}(n) :=
\left\{
1 - \ds\sum_{\ell \mid n \atop{\ell \in {\cal{P}}}}
(1 - g(\ell))
\right\}^{+}.
$$
Finally, we define the  weighted 
function
\begin{equation}\label{weighted-sifted-fcn-I-1.1}
H\left({\cal{A}}, {\cal{P}}, z^V, z^U\right)
:=
\ds\sum_{a \in {\cal{A}}}
{\cal{G}}\left(\gcd\left(a, P\left(z^U\right)\right)\right).
\end{equation}

 The weighted Grieves Lower Bound Sieve  provides a lower bound for 
 $H\left({\cal{A}}, {\cal{P}}, z^V, z^U\right)$.
 We will show in Lemma \ref{lemma-towards-lower} that,
 under additional hypotheses and with particular choices of $U$ and $V$,
 the growth of $H\left({\cal{A}}, {\cal{P}}, z^V, z^U\right)$, as a function of $z$,
 may be used to  estimate, from below, the number of almost-primes in ${\cal{A}}$.

To state the sieve, we require the following additional notation. 
For a positive integer $r$ and a positive real number $t$, we set
$$
h_{2r} (t)
:=
{\int
\ldots
\int}_{
\ds{t < t_{2r} < \ldots < t_1}
\atop{
\ds{3 t_{2 i} + \ldots + t_1 \geq 1 \ \forall 1 \leq i \leq r-1}
\atop{
\ds{3 t_{2r} + \ldots + t_1 \geq 1}
\atop{
\ds{t_{2 r} < 1 - t - t_1 - \ldots - t_{2r}}
}
}
}
}
\
\frac{1}{1 - t - t_1 - \ldots - t_{2r}}
\cdot
\frac{d t_1 \ldots d t_{2r} }{t_1 \ldots t_{2r}}
$$
and
$$
h(t) := \ds\sum_{r \geq 1} h_{2r}(t).
$$
For a positive real number $t$ with $t \leq \frac{1}{4}$, 
we set
$$
\psi(t) := \frac{1}{1-t} - h(t). 
$$
We recall from \cite[p. 205]{HaRi85}) that, for any $t \geq V_0$,
  \begin{equation}\label{positive-psi}
 \psi(t) \geq 0.
 \end{equation}
Finally, 
for $V$ satisfying ({\ref{U-V-one}) - (\ref{U-V-four}),
 we set
$$
\alpha(V) := \ds\int_{V}^{\frac{1}{4}} \psi(t) \ d t
$$
and
$$
\beta(V) := \ds\int_{V}^{\frac{1}{4}} \psi(t) \ \frac{d t}{t}.
$$

\bigskip

The following is the weighted Grieves Lower Bound Sieve \cite[Thm. A, p. 206]{HaRi85}.

\begin{theorem}\label{greaves-lower-bound-sieve}

Let ${\cal{A}}$ be a finite sequence of integers and let ${\cal{P}}$ be a set of primes.
Use the notation  ${\cal{A}}_{\ell}$, ${\cal{A}}_d$, $P(z)$, $S(\cal{A}, \cal{P}, z)$, and $W(z)$ 
introduced in Subsection  \ref{sec_sieve_setting-general}.
Let $U < V$ be such that (\ref{U-V-one}) - (\ref{U-V-four}) hold. 
With the above notation and under assumptions
(\ref{hypothesis-R-d}), 
(\ref{hypothesis-R-ell-squared}),
(\ref{w-function-asmptn-1}),
(\ref{w-function-asmptn-2}),
and
(\ref{A0-HRIp157-lower}),
there exist sequences of real numbers $(a_m)$, $(b_n)$ such that
$|a_m| \leq 1 \ \forall m$, 
$|b_n| \leq 1 \ \forall n$,
and such that,
for any real numbers $M, N$ satisfying the conditions
\begin{equation}\label{constraints-M-N}
M > z^U,
\quad
N > 1,
\quad
M N = z,
\end{equation}
provided that $z \to \infty$, we have
\begin{eqnarray}\label{lower-bound-H}
H\left({\cal{A}},  {\cal{P}}, z^V, z^U\right)
&\geq&
X \ W(z) \ \frac{2 e^{\gamma}}{U-V}
\left(
U \log \frac{1}{U}
+
(1-U) \log \frac{1}{1 - U}
\right.
\\
&&
\left.
-
\left(
\log \frac{4}{3}
-
\alpha(V)
\right)
- V \log 3
- V  \beta(V)
+
\O_{A, U}\left(\frac{\log \log \log z}{(\log \log z)^{\frac{1}{5}}}\right)
\right)
\nonumber
\\
&&
- (\log z)^{\frac{1}{3}}
\left|
\ds\sum_{m < M}
\ds\sum_{
n < N
\atop{
m n \mid P\left(z^U\right)
}
}
a_m b_n R_{m n}
\right|.
\nonumber
\end{eqnarray}
\end{theorem}

\bigskip

Note that the lower bound (\ref{lower-bound-H}) may be rewritten as
\begin{equation}\label{lower-bound-H-using-J}
\cal{H}(\cA,\cP, z^V, z^U) 
\geq
2e^\gamma X W(z) \left( J(U, V) 
+ 
\O_{A, U}\left( \frac{\log \log \log z}{(\log \log z)^{\frac{1}{5}}} \right) \right) 
- 
(\log z)^{\frac{1}{3}}
\bigg\lvert 
\ds\sum_{m<M} 
\ds\sum_{
n < N
\atop{
m n \mid P\left(z^U\right)
}
}
a_m b_n R_{mn}
\bigg\rvert ,
\end{equation}
where 	
\begin{equation}\label{definition-J}
J(U, V) 
:= 
\frac{1}{U - V}
\left(
 U \log \frac{1}{U} + (1-U) \log\frac{1}{1-U} - \log\frac{4}{3}+\alpha(V) - V \log3 -V \beta(V)
 \right).
\end{equation}
With this notation, 
we remark that, in order for (\ref{lower-bound-H-using-J}) to be meaningful, we need to choose parameters $U, V$ such that 
 $J(U, V) > 0$. 

\begin{remark}\label{J-calculation}
{\emph{
We recall from  \cite[p. 115]{DaWu12} that, 
for any $\frac{1}{6} \leq V \leq \frac{1}{4}$,
$J(U, V)$ may be studied 
using the simplified integral formulae 
\begin{equation}\label{alpha-DaWu}
\alpha(V)
=
\log \frac{4 (1 - V)}{3}
-
\ds\int_{4}^{\frac{1}{V}}
\left(
\frac{2}{u} \log (2 - u V)
+
\log \frac{1 - \frac{1}{u}}{1 - V}
\right)
\frac{\log (u - 3)}{u - 2}
\ d u,
\end{equation}
\begin{equation}\label{beta-DaWu}
\beta(V)
=
\log \frac{1 - V}{3 V}
-
\ds\int_{4}^{\frac{1}{V}}
\left(
\log (2 - u V) + \log \frac{1 - \frac{1}{u}}{1 - V}
\right)
\frac{\log (u - 3)}{u - 2}
\ d u.
\end{equation}
Numerical computations using these formulae
suggest that $J(U, V)$ is close to zero if $|1 - U - V| < 0.0005$.
In our application, we will make choices of $U$ and $V$  which ensure both the positivity of $J(U, V)$ and a balance between the magnitudes of the terms occurring on the right-hand side of (\ref{lower-bound-H-using-J}).
}}
\end{remark}


Now let us relate
$H\left({\cal{A}}, {\cal{P}}, z^V, z^U\right)$
to the number of almost-primes in ${\cal{A}}$.

\begin{lemma}\label{lemma-towards-lower}
Let ${\cal{A}}$ be a finite sequence of integers and let ${\cal{P}}$ be a set of primes.
Let $U < V$ be such that (\ref{U-V-one}) - (\ref{U-V-four}) hold. 
Assume that, for each $a \in \cA$, if a prime $\ell$ satisfies $\ell \mid a$, then $\ell \in \cP$.
Additionally, assume that
\begin{equation}\label{max-assumption}
\exists r \in \N \backslash \{0\}
\ \text{such that} \
\ds\max\left\{|a|: a \in {\cal{A}}\right\} \leq z^{r U + V}.
\end{equation}
Then
\begin{equation*}\label{Greaves-lemma-conc-omega}
\#\{ a\in \cA : \omega (a) \leq r \} \geq  H\left({\cal{A}}, {\cal{P}}, z^V, z^U\right)
\end{equation*}
and
\begin{equation*}\label{Greaves-lemma-conc-Omega}
\# \{ a\in \cA : \Omega(a) \leq r \} 
\geq 
H\left({\cal{A}}, {\cal{P}}, z^V, z^U\right)
+
\O\left(
\ds\sum_{z^V \leq \ell < z^U}
\left|
{\cal{A}}_{\ell^2}
\right|
\right).
\end{equation*}
\end{lemma}
\begin{proof}
We start by establishing the following two properties of $\cal{G}(n)$:
\begin{equation}\label{G-between-0-1}
0 \leq \cal{G}(n) \leq 1 \ \forall n
\end{equation}
and 
\begin{equation}\label{G-is-0}
\cal{G}(n) = 0 \ \forall n \ \text{such that} \ \gcd\left(n, P\left(z^V\right)\right) > 1.
\end{equation}

To prove (\ref{G-between-0-1}),
let $n$ be a non-zero natural number. We will consider several cases of $n$ according to its possible prime factors. 
 If $n$ is not divisible by any $\ell \in \cP$, then, from the definition of $\cal{G}$,  we get $\mathcal{G}(n) = 1$. 
 If $n$ is divisible by some $\ell \in \cP$ outside  the range $z^V \leq \ell \leq z^U$, 
 then, again from the definition of $\cal{G}$ , we get 
 $\mathcal{G}(n) = 0$. Now, fix $\ell \in \cP$ with $z^V \leq \ell < z^U$ and observe that
$0 \leq \frac{1}{U - V}\left( \frac{\log \ell}{\log z} - V \right) < 1$,
that is, 
$0 < 1 - g(\ell) \leq 1$.
Then, for any $n$ such that $\ell \mid n$,
we have
$0 \leq 1 - \ds\sum_{\substack{\ell \mid n \\ \ell \in \cP}} \Big( 1- g(\ell) \Big) < 1$,
from which we get that
 $0 \leq \cal{G}(n) < 1$.


To prove (\ref{G-is-0}),	
let $n$ be a non-zero natural number such that $\gcd\left(n, P\left(z^V\right)\right) > 1$.
Fix a prime $\ell \in \cP$ such that $\ell \mid n$ and $\ell < z^V$.
Observe that $g(\ell) = 0$, which implies that
$\ds\sum_{\substack{\ell \mid n \\ \ell \in \cP}} \Big( 1 - g(\ell) \Big) \geq 1$,
since we have shown that each summand is non-negative. 
Then $\cal{G}(n) = 0$, as claimed.
	
From (\ref{G-between-0-1}), we obtain that
\begin{equation}\label{Greaves-lemma-inequality}
\ds\sum_{\substack{a\in \cA \\ \cal{G}\left(\gcd\left(a, P\left(z^U\right)\right)\right) > 0}} 1 
\geq 
\ds\sum_{a \in \cA} \cal{G}\left(\gcd\left(a, P\left(z^U\right)\right)\right) 
= 
\cal{H}\left(\cA,\cP, z^V, z^U\right). 
\end{equation}

Note that, so far, we have not used  assumption (\ref{max-assumption}).
We claim that, under  assumption (\ref{max-assumption}),
 each integer $a$ counted on the left hand sum above satisfies  $\omega(a) \leq r$. 

To justify this claim, 
we fix $a \in \cA$ such that $\cal{G}\left(\gcd\left(a,P\left(z^U\right)\right)\right) > 0$
and deduce from (\ref{G-is-0})  that  $\gcd\left(a, P\left(z^V\right)\right) = 1$. 
Then
	\begin{align}
		0 	&< 1- \ds\sum_{\substack{\ell \mid a \\ \ell \leq z^U}}
				\left( 1- \frac{1}{U - V} \left( \frac{\log \ell}{\log z} - V \right) \right) \nonumber \\
			&= 1- \frac{1}{U - V}\ds\sum_{\substack{\ell \mid a \\ \ell \leq z^U}}
				\left(U - \frac{\log \ell}{\log z} \right) \nonumber \\
			&\leq 1- \frac{1}{U - V}\ds\sum_{\substack{\ell \mid a \\ \ell \leq z^U}}
				\left( U - \frac{\log \ell}{\log z} \right)
				- \frac{1}{U - V}\ds\sum_{\substack{\ell^k \mid a \\ \ell \geq z^U \\ k \geq 2}}
				\left(U - \frac{\log \ell}{\log z} \right) \nonumber \\
			&\leq 1 - \frac{U}{U - V}\cdot \omega\left(a; z^U\right) + \frac{1}{U - V}\cdot \frac{\log a}{\log z},\nonumber
	\end{align}
where we used the notation
\begin{equation*}
\omega(n; y) := \ds\sum_{\ell \mid n} 1 + \ds\sum_{\substack{\ell^k \mid n \\ \ell \geq y \\ k \geq 2}} 1.
\end{equation*}	
Following easy algebraic manipulations and invoking assumption (\ref{max-assumption}),
we obtain that
	\begin{equation*}
		U \cdot \omega\left(a; z^U\right) < U - V + \frac{\log a}{\log z} < U - V + (rU - V) = U (r+1),
	\end{equation*}
which gives  that $\omega\left(a; z^U\right) < r + 1$.
Since  $\omega(a) \leq \omega(a;z^u)$, we infer that $\omega(a) \leq r$, as desired.

We conclude that, under assumption (\ref{max-assumption}),
\begin{equation*}
\#\{ a\in \cA : \omega (a) \leq r \}
\geq
\ds\sum_{\substack{a\in \cA \\ \cal{G}\left(\gcd\left(a,P\left(z^U\right)\right)\right) > 0}} 1 
\geq 
 \cal{H}\left(\cA,\cP, z^V, z^U\right),
\end{equation*}
completing the first part of the lemma.
	
For the second part of the lemma, we start by rewriting the left hand sum of (\ref{Greaves-lemma-inequality}) as 
\begin{equation}\label{Greaves-lemma-sum-breakdown}
\ds\sum_{\substack{a\in \cA \\ \cal{G}\left(\gcd\left(a, P\left(z^U\right)\right)\right) > 0}} 1 
=
\ds\sum_{\substack{a\in \cA \\ \cal{G}\left(\gcd\left(a, P\left(z^U\right)\right)\right) > 0 \\ \Omega(a) = \omega\left(a,z^U\right)}} 1 \: \: 
+
\ds\sum_{\substack{a\in \cA \\ \cal{G}\left(\gcd\left(a, P\left(z^U\right)\right)\right) > 0 \\ \Omega(a) > \omega \left(a, z^U\right)}} 1. 
\end{equation}
On one hand, since we showed that each integer $a$ in the third sum above satisfies $\omega\left(a; z^U\right) \leq r $, 
we see that
 the first sum  is  smaller than $\#\{a\in \cA: \Omega(a) \leq r\}$. On the other hand, 
 for  an integer $a$ to be counted in the second sum, 
 since $\Omega(a) > \omega\left(a; z^U\right)$,
 there must be a prime $\ell < z^U$ such that $\ell^2 \mid a$. However, such a prime $\ell$ must also satisfy $\ell \geq z^V$, 
 because we showed in (\ref{G-is-0}) that $\cal{G}\left(\gcd\left(a,P\left(z^U\right)\right)\right)$ would be zero otherwise. Therefore,
\begin{equation}\label{Greaves-lemma-final}
\ds\sum_{
\substack{
a \in \cA \\ \cal{G}\left(\gcd\left(a, P\left(z^U\right)\right)\right) > 0 \\ \Omega(a) > \omega\left(a,z^U\right)
}
} 
1
\leq 
\#\left\{a \in \cA : \exists \ell \in \cP \text{ with } z^V \leq \ell < z^U \text{ and } \ell^2 \mid a\right\} 
\leq 
\ds\sum_{\substack{z^V \leq \ell < z^U \\ \ell \in \cP}} \# \cA_{\ell^2}. 
\end{equation}
By combining 
(\ref{Greaves-lemma-inequality}),
(\ref{Greaves-lemma-sum-breakdown}),
and 
(\ref{Greaves-lemma-final}),
we deduce that
\begin{equation*}
\#\left\{a \in {\cal{A}}: \Omega(a) \leq r\right\}
\geq
 H\left({\cal{A}}, {\cal{P}}, z^V, z^U\right)
 +
  \O\left(
 \ds\sum_{z^V \leq \ell < z^U}
 \left|
 {\cal{A}}_{\ell^2}
\right|
 \right),
\end{equation*}
completing the second part of the lemma.
\end{proof}

\medskip

We are now ready to apply Theorem \ref{greaves-lower-bound-sieve}
 to the study of the almost-prime Frobenius traces associated to an elliptic curve.
 Our particular sieve setting is the same as that used in the proof of Theorem \ref{main-thm-upper},
 hence the similarity between the first two paragraphs below and those in the beginning of the proof of
 our first main theorem.

\medskip

\begin{proof}[ {\bf{Proof of Theorem \ref{main-thm-lower} }}]
(i) We fix an elliptic curve $E$ defined over $\Q$, 
without complex multiplication, 
of conductor $N_E$, 
and 
of torsion conductor $m_E$.
We assume that
there exists  some $\frac{1}{2} \leq \theta < 1$
such that
the $\theta$-quasi-GRH holds for 
the Dedekind zeta functions of
$\Q(E[m])$ and $J_{E, m}$ for all positive integers $m$.
We fix $x > 2$, to be thought of as going to infinity.
As in Subsection \ref{sec_sieve_setting-ec},
we define
\begin{equation*}
\cA := \{a_p(E): p \leq x, p \nmid N_E, \gcd(a_p(E),m_E) = 1 \},
\end{equation*}
\begin{equation*}
\cP := \{\ell \: \text{prime}: \ell \nmid m_E \}.
\end{equation*}
We recall that $m_E$ is even, which  implies that all primes $\ell \in \cP$ are odd.

We showed in 
Subsection \ref{sec_sieve_setting-ec}
that 
\begin{equation*}\label{sieve-size-A-d}
\#{\cal{A}}_{d}
=
\frac{w(d)}{d}
X
+
R_d,
\end{equation*}
\begin{equation*}
\#{\cal{A}}_{\ell^2}
=
\frac{w(\ell^2)}{\ell^2}
X
+
R_{\ell^2},
\end{equation*}
where
\begin{equation*}
\label{definition-X}
X =
C_1(E)
 \pi(x),
\end{equation*}
 \begin{equation*}
 \label{definition-omega-d}
 w(d) 
 =
 \ds\prod_{\ell \mid d}
 \left(1 - \frac{1}{\ell^2}\right)^{-1},
\end{equation*}
\begin{equation*}\label{bound-R-d-upper}
|R_d| \ll_E d^2 x^{\theta} \log (d x),
\end{equation*}
\begin{equation*}\label{bound-R-ell-squared-upper}
\left|R_{\ell^2}\right| \ll_E \ell^4 x^{\theta} \log (\ell x).
\end{equation*}
We also showed that, for any $z > m_E$,
\begin{equation*}
W(z) 
=
C_2(E) \cdot
\left(
\frac{e^{-\gamma}}{\log z} + \o\left(\frac{1}{\log z}\right)
\right).
\end{equation*}	
Additionally, we showed that $w(\cdot)$ satisfies  assumptions
(\ref{hypothesis-R-d}),
(\ref{hypothesis-R-ell-squared}),
(\ref{w-function-asmptn-1}), 
 and 
 (\ref{w-function-asmptn-2}). 
Observing that, for any prime $\ell \in {\cal{P}}$, we have
$$
0 < w(\ell) = \left(1 - \frac{1}{\ell}\right)^{-1} = \frac{\ell^2}{\ell^2 -1} < \ell,
$$
we see that $w(\cdot)$ also satisfies (\ref{A0-HRIp157-lower}). 
Therefore, we may apply Theorem  \ref{greaves-lower-bound-sieve}, which we do below.

We fix an arbitrary positive real number $z > m_E$, to be thought of as going to infinity, 
we fix  $U, V$ satisfying  (\ref{U-V-one}) - (\ref{U-V-four}),
and
we fix $M, N$ satisfying (\ref{constraints-M-N}).
Using (\ref{lower-bound-H-using-J}),
we obtain that
\begin{equation}\label{lower-bound-H-using-J-applied}
\cal{H}(\cA,\cP, z^V, z^U) 
\geq
C(E) 
\frac{\pi(x)}{\log z} 
\left( J(U, V) 
+ 
\o(1)
\right) 
- 
(\log z)^{\frac{1}{3}}
\bigg\lvert 
\ds\sum_{m<M} 
\ds\sum_{
n < N
\atop{
m n \mid P\left(z^U\right)
}
}
a_m b_n R_{mn}
\bigg\rvert ,
\end{equation}
where
$C(E) = 2 C_1(E) C_2(E)$
and
where
 $J(U, V)$ is as in (\ref{definition-J}).
 
Remembering that $M N = z$ (from  (\ref{constraints-M-N}))
and that $|a_m| \leq 1 \ \forall m$, $|b_n| \leq 1 \ \forall n$ 
(from the setting of Theorem  \ref{greaves-lower-bound-sieve}), 
we deduce that
\begin{equation}\label{sum-R_m-lower}
\bigg\lvert 
\ds\sum_{m<M} 
\ds\sum_{\substack{n<N \\ mn \mid P(z^U)}}
a_m b_n R_{mn}
\bigg\rvert
\leq
\ds\sum_{\substack{d\leq z \\ d \mid P(z^U)}} 2^{\omega(d)}|R_d| 
\ll_E
\ds\sum_{d\leq z} 2^{\omega(d)} d^2 x^\theta \log (d x) 
\leq
x^\theta z^3 (\log x) (\log z).
\end{equation}

Our  goal is to make choices of $z, U$, and $V$  
that satisfy the following constraints:
they ensure that $J(U, V) > 0$;
they balance the growth of 
$\frac{\pi(x)}{\log z}J(U, V)$, as a function of $x$,
with the growth of
$\frac{x}{(\log x)^2}$;
and
they ensure that
$x^\theta z^3 (\log x) (\log z)^{1 + \frac{1}{3}} = \o\left(\frac{x}{(\log x)^2}\right)$.
Further restrictions on these parameters will emerge upon ensuring that
we are allowed to apply Lemma \ref{lemma-towards-lower} and upon seeking to minimize the 
number of prime factors of the integer $a_p(E)$.

To move towards a choice of $z$, 
we fix a positive real number $0 < \xi < 1$, which will be specified later, 
and set
\begin{equation}\label{x-in-terms-of-xi-x}
z := \frac{x^{\xi}}{(\log x)^2}.
\end{equation}
Then
(\ref{lower-bound-H-using-J-applied})
becomes
\begin{eqnarray}\label{lower-bound-H-using-J-applied-z}
\cal{H}\left(\cA,\cP, \frac{x^{\xi V}}{(\log x)^{2 V}}, \frac{x^{\xi U}}{(\log x)^{2 U}}\right) 
&\geq&
C(E)
\frac{\pi(x)}{
(\xi \log x - 2 \log \log x)
} 
\left( J(U, V) 
+ 
\o(1)
\right) 
\\
&&
- 
(\xi \log x - 2 \log \log x)^{\frac{1}{3}}
\bigg\lvert 
\ds\sum_{m<M} 
\ds\sum_{
n < N
\atop{
m n \mid P\left(\frac{x^{\xi U}}{(\log x)^{2 U}}\right)
}
}
a_m b_n R_{mn}
\bigg\rvert,
\nonumber
\end{eqnarray}
while
(\ref{sum-R_m-lower}) gives 
\begin{eqnarray*}
(\xi \log x - 2 \log \log x)^{\frac{1}{3}}
\bigg\lvert 
\ds\sum_{m<M} 
\ds\sum_{\substack{n<N \\ mn \mid P\left(\frac{x^{\xi U}}{(\log x)^{2 U}}\right)}}
a_m b_n R_{mn}
\bigg\rvert
&\ll_E&  
\frac{x^{3\xi + \theta}}{(\log x)^4}
(\xi \log x - 2 \log \log x)^{\frac{1}{3}}.
\nonumber
\end{eqnarray*}
Note that the function on the right hand side of the above inequality 
 is $\o \left(\frac{x}{(\log x)^2}\right)$ if
\begin{equation}\label{upper-bound-for-xi}
\xi \leq \frac{1-\theta}{3}.
\end{equation}
	
To ensure hypothesis (\ref{max-assumption}) of 
Lemma \ref{lemma-towards-lower}, 
we choose a positive integer $r$, to be specified later,
we recall that $\lvert a_p(E) \rvert < 2\sqrt{p} \leq 2\sqrt{x}$,
and 
we  note that assumption (\ref{max-assumption}) is satisfied if
\begin{equation*}
2\sqrt{x} \leq \left(\frac{x^\xi}{(\log x)^2}\right)^{rU+V}.
\end{equation*}
By examining the exponents of $x$ on each side, we see that this inequality holds if
$\frac{1}{2} < \xi (rU+V)$,
i.e., if
\begin{equation}\label{lower-bound-for-r}
r > \frac{1}{U}\left(\frac{1}{2\xi}-V\right).
\end{equation}
From this last inequality, we see that if any two of the three parameters $U$, $V$,  $\xi$ are held constant, 
then $r$ is minimized when the third parameter takes its largest possible value.

We choose $\xi$ as the largest possible, that is,
\begin{equation}\label{choice-xi}
\xi := \frac{1 - \theta}{3}.
\end{equation}
Then  the main term in (\ref{lower-bound-H-using-J-applied-z})
is
$\frac{  3 J(U, V) }{ 1 - \theta }  \text{\small{$C(E)$}} \frac{\pi(x)}{\log x}$.
Since we are seeking a main term as close as possible to 
$\text{\small{$C(E)$}} \frac{x}{(\log x)^2}$,
we want  $J(U, V) $  positive and as close as possible to $\frac{1-\theta}{3}$.
Recalling that $\frac{1}{2} \leq \theta < 1$, which is equivalent to
$0 < \frac{1- \theta}{3} \leq \frac{1}{6} = 0.166...$,
and recalling our initial restrictions
$0.074368... < V \leq \frac{1}{4}$,
$\frac{1}{2} \leq U$, $V < U$, $U + 3 V \geq 1$,
we choose
$$ U := 0.83, \; V := \frac{1}{6}.$$
In this case, $J\left(0.83,\frac{1}{6}\right) = 0.00692...$
and 
the minimal value of $r$ that we can choose is
$$r_1 := 1 + \left[ \frac{1}{0.83}\left(\frac{3}{2(1-\theta)} - \frac{1}{6} \right) \right].$$
Invoking  Lemma \ref{lemma-towards-lower}, we conclude  that
\begin{equation}\label{lower-bound-omega}
\#\{a_p(E): p\leq x, p \nmid N_E, \gcd(a_p(E), m_E) = 1, \omega(a_p(E)) \leq r_1 \}
\geq 
\frac{3}{1-\theta}(0.00692 + \o(1))C(E) \frac{\pi(x)}{\log x}.
\end{equation}
Note that removing the gcd condition $\gcd(a_p(E), m_E) = 1$
 will only make the set of primes $p$ larger.
 As such, we deduce that 
\begin{equation}\label{lower-bound-omega-no-gcd}
\#\{a_p(E): p \leq x, p \nmid N_E,  \omega(a_p(E)) \leq r_1 \}
\geq 
\frac{3}{1-\theta}(0.00692 + \o(1))C(E) \frac{\pi(x)}{\log x}.
\end{equation}
	

	
Now we focus on proving a similar result about $\Omega(a_p(E))$.
Using part (ii) of Proposition \ref{corollary-sizes-cond-sets}, we deduce that
\begin{align}\label{sum-A-ell^2}
\ds\sum_{
\substack{
z^V \leq \ell < z^U
\\
\ell \in \cP}} 
\# \cA_{\ell^2}
\ll_E
\ds\sum_{z^{V} \leq \ell < z^{U}} \left( \frac{\pi(x)}{\ell^2} + \ell^4 x^\theta\log x \right) 
\ll 
\frac{x^{1-\xi V}}{(\log x)^{1 - 2V}} + \frac{x^{5\xi U + \theta}}{(\log x)^{10U-1}}. 
\end{align}
Since $\xi$ and $V$ are both positive,  the first term 
on the right hand side inequality above is $\o\left(\frac{x}{(\log x)^2}\right)$. 
In order for the second term to also be $\o\left(\frac{x}{(\log x)^2}\right)$, 
we need
\begin{equation*}
		5\xi U + \theta \leq 1,
\end{equation*}
i.e.,
\begin{equation*}
U \leq \frac{1-\theta}{5\xi}.
\end{equation*}
Thus,  we may set
$$\xi := \frac{1-\theta}{3},$$
$$U := \frac{3}{5}.$$
In this case,  the right hand side of (\ref{lower-bound-for-r}) is minimized when we choose the largest possible $V$, i.e.,
$$V :=\frac{1}{4}.$$
Then  assumption (\ref{max-assumption}) from Lemma \ref{lemma-towards-lower} is satisfied for
\begin{equation*}
r_2 := 1 + \left[ \frac{5}{2(1-\theta)} - \frac{5}{12} \right].
\end{equation*}
Once again, with these choices, we obtain that the error term in 
(\ref{lower-bound-H-using-J-applied}) is negligible and  that $J(\frac{3}{5},\frac{1}{4}) = 0.3162... > 0$. 

Altogether, from the second part of Lemma \ref{lemma-towards-lower} we deduce that
\begin{equation}\label{lower-bound-Omega}
\#\{a_p(E): p\leq x, p \nmid N_E, \gcd(a_p(E),m_E) = 1, \Omega(a_p(E)) \leq r_2 \}
\geq 
\frac{3}{1-\theta}(0.3162 + \o(1))C(E) \frac{x}{(\log x)^2}.
\end{equation}
Since removing the gcd condition only makes the set larger, we obtain that
\begin{equation}\label{lower-bound-Omega-no-gcd}
\#\{a_p(E): p\leq x, p \nmid N_E,  \Omega(a_p(E)) \leq r_2 \}
\geq 
\frac{3}{1-\theta}(0.3162 + \o(1))C(E) \frac{x}{(\log x)^2}.
\end{equation}

We make one final remark.  
One may worry that the statements 
(\ref{lower-bound-omega-no-gcd})
and
(\ref{lower-bound-Omega-no-gcd})
are misleading,
 since they seem to offer lower bounds for the number of primes $p$ such that  the integer $a_p(E)$ is almost prime,
 while the primes $p$ being counted would also include those for which $a_p(E) = \pm 1$. 
 However, this inclusion does not impact the final result,  since, by  Lemma \ref{Lang-Trotter-bound-quasi-GRH},  $\#\left\{p\leq x: a_p(E) = \pm 1 \right\} \ll_E x^{1-\frac{1-\theta}{4}} = \o\left(\frac{x}{(\log x)^2}\right)$. 
Thus,  (\ref{lower-bound-omega-no-gcd}) and (\ref{lower-bound-Omega-no-gcd}) 
 give
\begin{equation*}\label{lower-bound-omega-non-units}
		\#\{a_p(E): p\leq x, p \nmid N_E, \gcd(a_p(E),m_E) = 1,  a_p(E) \neq \pm 1, \omega(a_p(E)) \leq r_1 \}
			\geq \frac{3}{1-\theta}(0.00692 + \o(1))C(E) \frac{x}{(\log x)^2}
	\end{equation*}
and
\begin{equation*}\label{lower-bound-Omega-non-units}
		\#\{a_p(E): p\leq x, p \nmid N_E, \gcd(a_p(E), m_E) = 1, a_p(E) \neq \pm 1, \Omega(a_p(E)) \leq r_2 \}
			\geq \frac{3}{1-\theta}(0.3162 + \o(1))C(E) \frac{x}{(\log x)^2}.
	\end{equation*}	
The proof of part (i) of Theorem \ref{main-thm-lower} is now complete.

\bigskip
(ii) We proceed as in part (i), with the exception of making the assumptions of GRH, AHC, and PCC. 
The effect of this change is in the estimates for $|R_d|$ and $|R_{\ell^2}|$, which, according to (\ref{upper-bound-R-d-PCC}) and (\ref{upper-bound-R-ell-squared-PCC}),
are
$$
|R_d| \ll_E  x^{\frac{1}{2}} \log (d x),
$$
$$
\left|R_{\ell^2}\right| \ll_E  x^{\frac{1}{2}}  \log (\ell x).
$$
In this case, the sum estimated in 
(\ref{sum-R_m-lower}) becomes
\begin{equation*}
\bigg\lvert 
\ds\sum_{m<M} 
\ds\sum_{\substack{n<N \\ mn \mid P(z^U)}}
a_m b_n R_{mn}
\bigg\rvert
\leq
\ds\sum_{\substack{d\leq z \\ d \mid P(z^U)}} 2^{\omega(d)}|R_d| 
\ll_E
x^{\frac{1}{2}} z (\log x) (\log z)
\end{equation*}
and the sum estimated in (\ref{sum-A-ell^2}) becomes
\begin{align}
\ds\sum_{
\substack{
z^V \leq \ell < z^U
\\
\ell \in \cP}} 
\# \cA_{\ell^2}
\ll_E
\ds\sum_{z^{V} \leq \ell < z^{U}} \left( \frac{\pi(x)}{\ell^2} +  x^{\frac{1}{2}}\log x \right) 
\ll 
\frac{\pi(x)}{V z^V}  + \frac{x^{\frac{1}{2}} z^U \log x}{U }.
\nonumber
\end{align}

We choose
$$
z := \frac{x^{\frac{1}{2}}}{(\log x)^5},
$$
$$
U := 0.5111286..., \ V := \frac{1}{4},
$$
i.e. $U$ is chosen to be the solution to
$J(U, \frac{1}{4}) = \frac{1}{2}$. In this way,
we deduce that
$$
{\cal{H}}\left(
{\cal{A}}, {\cal{P}}, z^V, z^U
\right)
\geq
C(E)  \frac{\pi(x)}{\log x} \left(1 + \o(1) \right) + \o \left(\frac{x}{(\log x)^2}\right).
$$
The above choices of $z, U, V$ also imply that
$$
\ds\sum_{
\substack{
z^V \leq \ell < z^U
\\
\ell \in \cP}} 
\# \cA_{\ell^2}
=
\o \left(\frac{x}{(\log x)^2}\right).
$$
Moreover, hypothesis (\ref{max-assumption}) of 
Lemma \ref{lemma-towards-lower} is ensured if 
$$r = 2.$$
Therefore
\begin{equation*}
		\#\{a_p(E): p\leq x, p \nmid N_E, \gcd(a_p(E), m_E) = 1, a_p(E) \neq \pm 1, \Omega(a_p(E)) \leq 2 \}
			\geq (1 + \o(1))C(E) \frac{x}{(\log x)^2}
\end{equation*}	
The proof of part (ii) of Theorem \ref{main-thm-lower} is now complete.
\end{proof}

\section{Concluding remarks and further questions}

In the setting of an elliptic curve $E$ defined over $\Q$, of conductor $N_E$, without complex multiplication, and of torsion conductor $m_E$,
from the Chebotarev density theorem
we know that
a positive set of primes $p$ have the property that $a_p(E)$ is even.
For these primes, instead of studying the primality of $a_p(E)$, we can study the primality of $\frac{a_p(E)}{2}$.
For example, it is conjectured in \cite{CoJo22}
 that, similarly to (\ref{conjecture-function}) and (\ref{conjecture-constant}),
 we have
 \begin{equation}\label{conjecture-function-bis}
\#
\{
p \leq x:
p \nmid N_E,
\frac{|a_p(E)|}{2}  \ \text{is prime}
\}
\sim C'(E) \frac{x}{(\log x)^2},
\end{equation}
 where
 \begin{equation}\label{conjecture-constant-bis}
C'(E) 
:= 
\frac{m_E}{\phi(m_E)}
\cdot
\frac{
\#\left\{
M \in \Gal(\Q(E[2 m_E])/\Q): 
\tr M \in 2(\Z/2 m_E \Z)^{\times}
\right\}
}{
\#\Gal(\Q(E[2 m_E])/\Q)}
\cdot
\ds\prod_{\ell \nmid m_E \atop{\ell \ \text{prime}}} 
\left(
1 - \frac{1}{\ell^3 - \ell^2 - \ell + 1}
\right).
\end{equation}
The proofs of Theorems \ref{main-thm-upper}, Corollary \ref{Brun-analog}, and Theorem \ref{main-thm-lower}
can  be easily adapted to also prove results about the (almost) primality of the integers $\frac{a_p(E)}{2}$.
With more work, the methods of proofs for these results can also be adapted to obtain results about the 
(almost) primality of the Frobenius traces associated to
an arbitrary  product of non-isogenous elliptic curves defined over $\Q$ 
and without complex multiplication, 
and 
to  an arbitrary abelian variety defined over $\Q$, of dimension $g$, whose adelic Galois representation 
has open image in the symplectic group $\GSp_{2g}(\hat{\Z})$. This more complex adaptation is relegated to a future project.

In the setting of an elliptic curve $E$ defined over $\Q$, of conductor $N_E$, 
and with complex multiplication by some order ${\cal{O}}$ in an imaginary quadratic field $K$,
all primes $p \nmid N_E$ of ordinary reduction for $E$ have the property that
$\Q\left(\sqrt{a_p(E)^2 - 4p}\right) \simeq K$.
This special property of $E$ leads to a tight connection between
the study of the primality of $a_p(E)$ (or of $\frac{a_p(E)}{2}$)
and the study of the primes $p$ for which there exist a prime $\ell$ and an integer $n$ such that
$p = Q_{E, {\cal{O}}}(\ell, n)$,
where
$Q_{E, {\cal{O}}}(\ell, n)$ is a particular
positive definite, primitive, integral, binary quadratic form,
defined by $E$ and by the imaginary quadratic order ${\cal{O}}$.
As explained in the upcoming notes \cite{CoJo22} and \cite{Jo22}, 
the analogues of the asymptotic formulae 
(\ref{conjecture-function})
and
(\ref{conjecture-function-bis})
for an elliptic curve $E$ with complex multiplication 
agree with 
the asymptotic formula
(\ref{counting-Q(ell, n)})
for $Q_{E, {\cal{O}}}$.

In light of the above remark,  
an ambitious research project is  that of adapting some of the ideas and methods
of \cite{FoIw97}, \cite[Thm. 18.6]{FrIw10},  and \cite{LaScXi20}
to study 
conjectures (\ref{conjecture-function}) and (\ref{conjecture-function-bis})
in the case of an elliptic curve $E$ without complex multiplication.

We conclude by posing the following question.

{\bf{Question}}

Let $E$ be an elliptic curve defined over $\Q$, of conductor $N_E$, and without complex multiplication.
Let $K$ be an imaginary quadratic field, of discriminant $d_K$.
Define $t_K$ to be $1$ if $2 \nmid d_K$ and to be $2$ if $2 \mid d_K$.
Are there infinitely many primes $p \nmid N_E$ 
such that $\Q\left(\sqrt{a_p(E)^2 - 4p}\right) \simeq K$
and
$\frac{|a_p(E)|}{t_K}$ is prime?


{\small{

}}

\end{document}